\documentclass[review,hidelinks,onefignum,onetabnum]{siamart220329}

\newcommand{\divt}{\operatorname{div}}

\newcommand{\be}{\boldsymbol{\varepsilon}}

\newcommand{\bft}{\boldsymbol{f}}  %

\newcommand{\bn}{\boldsymbol{n}}

\newcommand{\bw}{\boldsymbol{w}}
\newcommand{\bu}{\boldsymbol{u}}

\newcommand{\bv}{\boldsymbol{v}}

\newcommand{\bK}{\boldsymbol{K}}

\newcommand{\bM}{\boldsymbol{M}}

\newcommand{\bN}{\boldsymbol{\sigma}}

\newcommand{\bV}{\boldsymbol{V}}

\newcommand{\bvar}{\boldsymbol{\varepsilon}} %
\newcommand{\bsig}{\boldsymbol{\sigma}}

\newcommand{\btau}{\boldsymbol{\tau}}

\makeatletter
\newcommand*{\rom}[1]{\expandafter\@slowromancap\romannumeral #1@}
\makeatother

\usepackage{lipsum}
\usepackage{amsfonts}
\usepackage{graphicx}
\usepackage{epstopdf}
\usepackage{algorithmic}
\usepackage{mathtools}
\usepackage{amsmath,amssymb}
\usepackage{cite}

\ifpdf
  \DeclareGraphicsExtensions{.eps,.pdf,.png,.jpg}
\else
  \DeclareGraphicsExtensions{.eps}
\fi

\newsiamremark{remark}{Remark}
\newsiamremark{hypothesis}{Hypothesis}
\crefname{hypothesis}{Hypothesis}{Hypotheses}
\newsiamthm{claim}{Claim}

\title{The Coupling of Mixed and Primal Finite Element Methods for the Coupled Body-Plate Problem\thanks{The manuscript has been submitted to editors.\funding{The first author was supported by the National Natural Science Foundation of China under grant No. 12288101. The third author was supported by NSFC project No. 12301466 and by NSFC project No. 12422114.}}}

\author{Jun Hu \thanks{LMAM and School of Mathematical Sciences, Peking University, Beijing 100871, P. R. China. Chongqing Research Institute of Big Data, Peking University, Chongqing 401332, P. R. China. (\email{hujun@math.pku.edu.cn}).}
\and Zhen Liu \thanks{LMAM and School of Mathematical Sciences, Peking University, Beijing 100871, P. R. China. Chongqing Research Institute of Big Data, Peking University, Chongqing 401332, P. R. China. (\email{zliu37@pku.edu.cn}).} 
\and Rui Ma \thanks{School of Mathematics and Statistics, Beijing Institute of Technology, Beijing 100081, P. R. China. (\email{rui.ma@bit.edu.cn}).}}

\usepackage{amsopn}

\ifpdf
\hypersetup{
  pdftitle={A New Mixed Finite Element Method For The Cahn-Hilliard Equation},
  pdfauthor={Zhen Liu, Rui Ma}
}
\fi
\usepackage{booktabs}
\usepackage{mathtools}
\usepackage{stmaryrd}
\SetSymbolFont{stmry}{bold}{U}{stmry}{m}{n}
\usepackage{tikz}
\usepackage{graphicx}
\usepackage{subcaption}
\usepackage{hyperref}

\begin{document}

\maketitle
\begin{abstract}
This paper considers the coupled problem of a three-dimensional elastic body and a two-dimensional plate, which are rigidly connected at their interface. The plate consists of
a plane elasticity model along the longitudinal direction and a plate bending model with Kirchhoff assumptions along the transverse direction. The Hellinger–Reissner formulation is adopted for the body by introducing the stress as an auxiliary variable, while the primal formulation is employed for the plate. The well-posedness of the weak formulation is established. This approach enables direct stress approximations and allows for non-matching meshes at the interface since the continuity condition of displacements acts as a natural boundary condition for the body. Under certain assumptions, discrete stability and error estimates are derived for both conforming and nonconforming finite element methods. Two specific pairs of conforming and nonconforming finite elements are shown to satisfy the required assumptions, respectively. Furthermore, the problem is reduced to an interface problem based on the domain decomposition, which can be solved effectively by a conjugate gradient iteration. Numerical experiments are conducted to validate the theoretical results.
\end{abstract}


\section{Introduction}

This paper focuses on the problem of the coupled body-plate, a significant structure in engineering applications such as H-shaped beams embedded in elastic foundations or satellite designs \cite{Ciarlet1989JMPA}. Mathematical descriptions of the body-plate model with various model assumptions have been proposed, see, e.g., \cite{Ciarlet1989JMPA, Fengshi1996, jianguo2005some, arango1998some, wang1992mathematical}. In this paper, the body is modeled as a three-dimensional linear elastic structure, while the plate is composed of a plane elasticity model along the longitudinal direction and a plate bending model with Kirchhoff assumptions along the transverse direction \cite{jianguo2005some}. These two components are coupled by rigid junction conditions.

Since the Kirchhoff plate model requires $C^1$-continuous finite element discretizations, preserving conformity along the interface poses a significant challenge, even when conforming elements are utilized for the plate discretization \cite{aufranc1989numerical, jianguo2005some}. Huang et al. established discrete stability and provided error estimates by employing nonconforming finite element methods to discretize the Kirchhoff plate model \cite{huang2007finite, huang2006finite, huang2011NMPDE, chen2009ap}. Although displacement-based finite element methods are easy to implement, they exhibit limitations in enforcing junction conditions efficiently, especially when meshes are non-matching.
Some additional treatments are required to deal with the coupling with non-matching meshes, e.g., using intermediate finite elements \cite{bournival2010mesh, klarmann2022coupling}, Nitsche’s method \cite{hansbo2005nitsche, hansbo2022nitsche, nguyen2014nitsche, yamamoto2019numerical}, static condensation \cite{shim2002mixed, song2010rigorous}, or mortar approaches\cite{ljulj20193d, ljulj2021numerical, steinbrecher2020mortar}.

In \cite{hu2024direct, wieners1998coupling}, authors used mixed finite element methods in one subdomain and conforming finite elements in the other for solving the problem on two disjoint subdomains. The method allows for non-matching meshes on the interface, eliminating the need for strict conformity and enhancing computational efficiency. Motivated by this idea, this paper establishes a new mixed weak formulation for the mixed-dimensional problem: the coupled body-plate model. Specifically, by introducing the stress as an independent variable in $H(\divt;\mathbb{S})$, the space of square-integrable symmetric tensors with square-integrable divergence, this paper utilizes the Hellinger-Reissner variational principle \cite{ArnoldWinther2002} for the body, while retaining the primal formulation for the plate. A dual product in different dimensions is rigorously defined by extending plate functions to the boundary of the body. Based on this dual product, the well-posedness of the weak formulation is proved, and its equivalence with the displacement-based formulation is provided. Consequently, the continuity condition of displacements poses a natural boundary condition for the mixed finite element method on the body. Furthermore, the proposed method yields direct approximations of the body stress; see \cite{hu2025mixed} for a mixed formulation of the coupled plates model, where stresses and moments were treated as independent variables.  

Under appropriate assumptions, this paper develops a general framework for finite element methods, together with discrete stability analysis and error estimates for both conforming and nonconforming schemes. As an example of conforming finite element methods, this paper provides the conforming $H(\divt;\mathbb{S})$ element in \cite{Hu2015} for the body, the vectorial Lagrange element for the plane elasticity model and the Argyris element \cite{Argyris1968}  for the plate bending model. Analogously, 
the nonconforming $H(\divt;\mathbb{S})$ element in \cite{HuMa2018}, the vectorial Lagrange element and the Morley element are used to give an example of nonconforming finite element methods. It is noteworthy that the discrete stability analysis and error estimates for the nonconforming method remain valid on non-matching meshes without imposing restrictive conditions. The proof relies on an inequality that bounds the dual product between a nonconforming $H(\divt; \mathbb{S})$ function and an $H^1$ function, which constitutes a new theoretical contribution of independent interest. Based on this inequality, an interpolation operator is introduced to map $H^2$ nonconforming finite element spaces to continuous spaces on the interface, following ideas from \cite{arnold1985mixed, veeser2019quasi}. The new weak formulation allows using tetrahedral or cubic meshes for the body and triangular or rectangular meshes for the plate, without requiring mesh type consistency between the interface. Interested readers can refer to \cite{Hu20152, hu2014simple} and \cite{adini1960analysis, arnold2005rectangular} for conforming $H(\divt;\mathbb{S})$ finite elements and $C^1$ continuous elements, respectively and refer to \cite{arnold2014nonconforming, gopalakrishnan2011symmetric, hu2016simplest} for nonconforming $H(\divt;\mathbb{S})$ finite elements. Numerical experiments with matching and non-matching meshes at the interface are presented to validate theoretical results. Furthermore, this paper provides a conjugate gradient iteration for the reduced interface problem of the coupled body-plate model and shows its efficiency, analogously as in \cite{lazarov2001iterative}. Non-overlapping domain decomposition methods for the displacement-based formulation can be found in \cite{huang2004numerical}.
 
The rest of this paper is organized as follows: Section 2 presents notations, the mathematical model, and the displacement-based formulation for the coupled body-plate model. The new mixed weak formulation and its well-posedness are established in Section 3. In Section 4,  frameworks of conforming and nonconforming finite element methods  and their examples are provided, respectively. Under some assumptions, the discrete stability and error estimates are proved. Some numerical experiments are presented in Section 5 to justify our theoretical results and the effectiveness of the proposed method.


\section{Preliminaries}
\label{Pre}

This section presents the model assumptions and notation conventions in this paper. The displacement-based formulation for the coupled body-plate model is recalled for subsequent discussions.

\subsection{Hypotheses and notations}

This paper considers the coupled problem of a body $\alpha$ and a plate $\beta$ with a rigid junction $\Gamma$, which is shown in Figure \ref{figmodel}. In this paper, $\alpha$ is a polyhedron in $\mathbb{R}^3$ and $\beta$ is a polygon in $\mathbb{R}^2$. The global right-handed orthogonal coordinate $(x_1,x_2,x_3)$ is fixed in the space with $x_1$ and $x_2$ the longitudinal directions of $\beta$, and $x_3$ the transverse direction of $\beta$. 
Let $\bn^\alpha$ and $\bn^\beta$ denote the unit outward normal vector of $\partial \alpha$ and $\partial \beta$, respectively, and let $\partial_{n^{\beta}}(\bullet) \coloneqq \nabla(\bullet) \cdot \bn^\beta$.

\begin{figure}[ht]
\centering
\begin{tikzpicture}[scale=1.5]
\draw[-stealth, thick] (-1.8,-0.5,1) -- (-1.8+0.7,-0.5,1) node[pos=0.7, above]{$x_2$};
\draw[-stealth, thick] (-1.8,-0.5,1) -- (-1.8,-0.5+0.7,1) node[pos=1, left]{$x_3$};
\draw[-stealth, thick] (-1.8,-0.5,1) -- (-1.8,-0.5,1+0.7) node[pos=1, right]{$x_1$};
\draw[thick] (-2.1,-0.5,-1) -- (-1.25,-0.5,-1);
\draw[dashed] (-1.25,-0.5,-1) -- (0.1,-0.5,-1);
\draw[thick] (0.1,-0.5,-1) -- (2.1,-0.5,-1);
\draw[thick] (-2.1,-0.5,2) -- (2.1,-0.5,2);
\draw[thick] (2.1,-0.5,2)-- (2.1,-0.5,-1);
\draw[thick] (-2.1,-0.5,-1)-- (-2.1,-0.5,2);
\node at (1.8,-0.5,-0.6) {\(\beta\)}; 
\coordinate (A) at (-0.5,-0.5,0);
\coordinate (B) at (0.5,-0.5,0);
\coordinate (C) at (0.5,0.5,0);
\coordinate (D) at (-0.5,0.5,0);
\coordinate (A1) at (-0.5,-0.5,1);
\coordinate (B1) at (0.5,-0.5,1);
\coordinate (C1) at (0.5,0.5,1);
\coordinate (D1) at (-0.5,0.5,1);
\fill[gray!20,opacity=0.5] (A) -- (A1) -- (B1) -- (B) -- cycle;
\draw[dashed] (A) -- (B) -- (C) -- (D) -- cycle; 
\draw[dashed] (A) -- (A1); 
\draw[thick]  (B) -- (C);
\draw[thick]  (C) -- (D);
\draw[thick] (A1) -- (B1) -- (C1) -- (D1) -- cycle; 
\draw[thick] (B1) -- (B); 
\draw[thick] (C1) -- (C); 
\draw[thick] (D1) -- (D);
\node at (0,0.3,0) {\(\alpha\)};
\node at (0,-0.6,0.3) {\(\Gamma\)};
\end{tikzpicture}
\caption{Coupled body-plate with a global coordinate.}
\label{figmodel}
\end{figure}
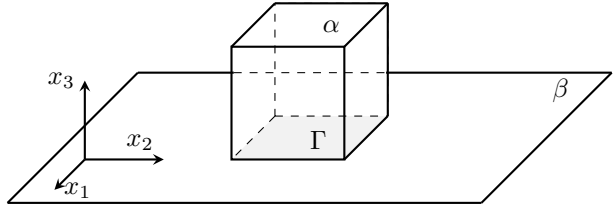

Given a bounded domain $G$, let $H^{m}(G;X)$ denote the Sobolev space of functions of $L^{2}(G; X)$ within domain $G$, taking values in space $X$, whose distributional derivatives up to the order $m$ also belong to $L^{2}(G; X)$. When $m=0, H^0(G;X)$ is exactly $L^2(G;X)$. Let $C^m(G; X)$ denote the space of $m$-times continuously differentiable functions and $P_{l}(G; X)$ denote the set of all the polynomials with the total degree no greater than $l$. In this paper, $X$ could be $\mathbb{R}$, $\mathbb{R}^{2}$, $\mathbb{R}^{3}$ or $\mathbb{S}$, where $\mathbb{S}$ denotes the set of all symmetric $\mathbb{R}^{3\times 3}$ matrices. If $X=\mathbb{R}$, then $H^{m}(G)$ abbreviates $H^{m}(G; X)$, similarly for $C^m(G), P_{l}(G)$. The standard Sobolev norm $\|\bullet\|_{m,G}$ and semi-norm $|\bullet|_{m,G}$ will be taken. The $L^2$ inner product is denoted by $(\bullet, \bullet)_G$ for the scalar, vector-valued, and tensor-valued $L^2$ spaces over $G$. Specially, let $\langle \bullet, \bullet \rangle_{\partial \alpha}$ denote the duality pairing between $H^{\frac{1}{2}}(\partial \alpha;\mathbb{R}^3)$ and its dual $H^{-\frac{1}{2}}(\partial \alpha;\mathbb{R}^3)$. Define the following Hilbert spaces:
\begin{equation}
\label{def:norms}
\begin{aligned}
H(\divt,\alpha; \mathbb{S}) &\coloneqq  \{ \btau^\alpha \in L^2(\alpha;\mathbb{S}); \divt \btau^\alpha \in L^2(\alpha;\mathbb{R}^3) \}, \\
H^1_0(\beta) & \coloneqq \{  v^{\beta} \in H^1(\beta);  v^{\beta} = 0 \text{ on } \partial \beta \},\\
H^{2}_0(\beta) &\coloneqq \{ v^{\beta} \in H^2(\beta);  v^{\beta} = 0,  \partial_{n^{\beta}} v^{\beta} =0 \text{ on } \partial \beta  \},\\
\end{aligned}
\end{equation}
each equipped with the norm $\|\btau^\alpha\|_{\divt,\alpha}^2\coloneqq \|\btau^\alpha\|_{0,\alpha}^2+\|\divt \btau^\alpha\|_{0,\alpha}^2$, $|v^\beta|_{1,\beta}$ and $|v^\beta|_{2,\beta}$, respectively. 

Throughout this paper, $C$ denotes a generic positive constant that is independent of the functions involved in the inequality and the mesh parameters of the triangulation. Note that $C$ can take different
values in different places. 

\subsection{Mathematical model}
This paper assumes the linear behavior of elastic material in $\alpha$, the plane elasticity model in the longitudinal direction of $\beta$, and the plate bending model in the transverse direction of $\beta$. 

Let $E_\omega> 0, \nu_\omega\in ( 0, \frac 12) $ denote Young's modulus, Poisson's ratio of the elastic member $\omega= \alpha, \beta$, respectively, $t_{\beta}$ be the thickness of plate $\beta$, $\delta_{ij}$ and $\delta_{IJ}$ stand for the usual Kronecker delta. Unless stated to the contrary, Latin indices $i, j, l$ take their values in the set ${1, 2, 3}$, while the capital Latin indices $I, J, L$ take their values in the set ${1,2}$. The summation convention whereby summation is implied when a Latin index (or a capital Latin index) is repeated exactly two times.

Given the displacement field $\bv^\alpha \in\mathbb{R}^3$ in $\alpha$, the strain $\bvar^\alpha(\bv^\alpha)$, the stress $\bsig^\alpha(\bv^\alpha)$ and its constitutive equation in $\alpha$ are
\begin{equation*}
	\begin{aligned}
		&\varepsilon_{ij}^\alpha({\bv^\alpha})\coloneqq( \partial v_j^\alpha/\partial x_i + \partial v_i^\alpha/\partial x_j)/2,\\
		&\sigma_{ij}^\alpha(\bv^\alpha)\coloneqq \frac{E_\alpha}{1+\nu_\alpha}\varepsilon_{ij}^\alpha(\bv^\alpha)+\frac{E_\alpha\nu_\alpha}{(1+\nu_\alpha)(1-2\nu_\alpha)}\varepsilon_{ll}^\alpha(\bv^\alpha)\delta_{ij} = \mathcal{C}_0 \varepsilon_{ij}^{\alpha}(\bv^\alpha). \\
		\end{aligned}
\end{equation*}
Given the displacement field  $\bv^\beta\in\mathbb{R}^3$ in $\beta$, the strain $\bvar^\beta(\bv^\beta)$, the stress $\bN^\beta(\bv^\beta)$ and its constitutive equation in $\beta$ for the plane elasticity model are
\begin{equation*}
\begin{aligned}
		&\varepsilon_{IJ}^\beta(\bv^\beta)\coloneqq (\partial v_J^\beta/\partial x_I+ \partial v_I^\beta/ \partial x_J )/2,\\
		&\sigma_{IJ}^{\beta}(\bv^\beta)\coloneqq \frac{E_{\beta} t_{\beta}}{1-\nu_{\beta}^{2}}((1-\nu_{\beta})\varepsilon_{IJ}^{\beta}(\bv^\beta)+\nu_{\beta}\varepsilon_{LL}^{\beta}(\bv^\beta)\delta_{IJ})=\mathcal{C}_1 \varepsilon_{IJ}^{\beta}(\bv^\beta), \\
\end{aligned}
\end{equation*}
the curvature $\bK^\beta(\bv^\beta)$, the moment $\bM^\beta(\bv^\beta)$ and its constitutive equation in $\beta$ for plate bending model are
\begin{equation}
\label{def:KandM}
\begin{aligned}
		&K_{IJ}^\beta(\bv^\beta)\coloneqq - {\partial^2v_3^\beta}/{\partial x_I\partial x_J}, \\
		&M_{IJ}^\beta(\bv^\beta)\coloneqq \frac{E_\beta t_\beta^3}{12(1-\nu_\beta^2)}( (1-\nu_\beta)K_{IJ}^\beta(\bv^\beta)+\nu_\beta K_{LL}^\beta(\bv^\beta) \delta_{IJ})=\mathcal{C}_2 K_{IJ}^{\beta}(\bv^\beta).
\end{aligned}
\end{equation}
Interested readers can refer to \cite{jianguo2005some, huang2006finite} for details.

For simplicity, in this paper, let $\partial \alpha \backslash \Gamma$ be the free boundary and $\partial \beta$ be the clamped boundary. Let $\Omega \coloneqq \{\alpha, \beta \}$ denote the whole domain. Given the load field $(\bft^{\alpha}, \bft^{\beta})$, the displacement field $(\bu^{\alpha}, \bu^{\beta})$ of the equilibrium configuration of $\Omega$ is concerned. The equilibrium equations \cite{arango1998some, jianguo2005some} of the coupled body-plate model can be written as
\begin{equation}
\label{equil}
\begin{aligned}
&& && -\divt \bsig^{\alpha}(\bu^\alpha) &= \bft^{\alpha}& &\text{ in } \alpha, &&&   \\
&& &&  - \partial \sigma_{IJ}^{\beta}/\partial x_J(\bu^\beta)  & =f_{I}^{\beta} & &\text{ on } \beta \backslash \Gamma , &&& \\
&& && - \partial^2 M_{IJ}^{\beta} /\partial x_{I}\partial x_{J} (\bu^\beta)&= f_{3}^{\beta} &&\text{ on } \beta \backslash \Gamma,&&& 
\end{aligned}    
\end{equation}
with boundary conditions 
\begin{equation}
\label{bd}
\begin{aligned}
&&  && \bsig^{\alpha}(\bu^\alpha) \bn^{\alpha}&=0   && \text{ on } \partial \alpha \backslash \Gamma, &&&   \\
&& && \bu^{\beta} = 0, \, \partial_{n^{\beta}} u_3^{\beta} &=0 & & \text{ on } \partial \beta, &&& 
\end{aligned}
\end{equation}
and junction conditions
\begin{subequations}
\label{jun}
\begin{align}
&& && \bu^\alpha &= \bu^\beta && \text{ on } \Gamma, &&&\label{jun1}  \\
&& && - \partial \sigma_{IJ}^{\beta}/ \partial x_J(\bu^\beta) &= f_{I}^{\beta} -	\sigma_{Ij}^{\alpha}(\bu^\alpha) n_j ^{\alpha} &&\text{ on } \Gamma,  &&&\label{jun2}  \\
 &&&&  -  \partial^2 M_{IJ}^{\beta} /\partial x_{I}\partial x_{J}(\bu^\beta) &= f_{3}^{\beta}-	\sigma_{3j}^{\alpha}(\bu^\alpha) n_j ^{\alpha} &&\text{ on } \Gamma. &&&\label{jun3}
\end{align}    
\end{subequations}
In this paper, the local coordinate systems on $\alpha$ and $\beta$ are assumed to coincide with the global coordinate system. However, the proposed method remains applicable when $\alpha$ and $\beta$ adopt distinct local coordinates. In such cases, the junction conditions \eqref{jun} are no longer expressed componentwise but involve transformation coefficients that relate the respective coordinate systems.

\subsection{The displacement-based formulation} This subsection recalls the variational formulation based on displacements from \cite{jianguo2005some}. 
Define the spaces for displacements in $\alpha$ and $\beta$ as follows:
\begin{equation}
\label{def:W-beta}
W^\alpha \coloneqq H^1(\alpha;\mathbb{R}^3), \, W^\beta \coloneqq W_{1}^\beta \times W_2^\beta \coloneqq H_0^1(\beta;\mathbb{R}^2) \times H^2_0(\beta). 
\end{equation}
Introduce the following space
\begin{equation}
\label{def: bV}
\bV \coloneqq \{ (\bv^\alpha, \bv^\beta) \in W^\alpha \times W^\beta;\,
\bv^{\alpha} = \bv^{\beta} \text{ on } \Gamma  \}.    
\end{equation}
The displacement-based formulation for the coupled body-plate model reads: Given $\bft^\alpha \in L^2(\alpha;\mathbb{R}^3)$ and $\bft^\beta \in L^2(\beta;\mathbb{R}^3)$, find $(\bu^\alpha,\bu^\beta) \in \bV$ such that 
\begin{equation}
\label{pro:displacement-based}
\begin{aligned}
&(\bsig^\alpha(\bu^\alpha),\bvar^\alpha(\bv^\alpha))_\alpha + (\bsig^\beta(\bu^\beta), \bvar^\beta(\bv^\beta))_\beta+ (\bM^\beta(\bu^\beta),\bK^\beta(\bv^\beta))_\beta \\
& = ({\bft}^\alpha, \bv^\alpha)_\alpha + (\bft^\beta, \bv^\beta )_\beta  \quad \text{ for all } (\bv^\alpha,\bv^\beta) \in \bV.    
\end{aligned}
\end{equation}

\begin{theorem}[{\cite[Theorem 2.2]{jianguo2005some}}]
Given $\bft^\alpha \in L^2(\alpha;\mathbb{R}^3)$ and $\bft^\beta \in L^2(\beta;\mathbb{R}^3)$, then there exists a unique function $(\bu^\alpha, \bu^\beta) \in \bV$ satisfying \eqref{pro:displacement-based}.
\end{theorem}


\section{The mixed weak formulation}

This section introduces a new weak formulation for the coupled body-plate model. The weak formulation employs a mixed formulation for the body $\alpha$ by introducing the stress $\bsig^\alpha$ as an auxiliary variable and a primal formulation for the plate. 

Let $V^\alpha \coloneqq L^2(\alpha;\mathbb{R}^3)$. Define $H_{00}^{\frac{1}{2}}(\partial \alpha \backslash \Gamma; \mathbb{R}^3)$ as the set of $\bv^\alpha \in H^{\frac{1}{2}}(\partial \alpha \backslash \Gamma; \mathbb{R}^3)$ such that its zero extension to $\partial \alpha$ belongs to $H^{\frac{1}{2}}(\partial \alpha ;\mathbb{R}^3)$ \cite{lions2012non}. Define
\begin{equation}
\label{defSigmaalpha}
	\begin{aligned}
	\Sigma^{\alpha} \coloneqq  \{ \btau^{\alpha} \in H(\divt,\alpha; \mathbb{S}); \langle \btau^{\alpha}  \bn^\alpha, \bv^\alpha \rangle_{\partial \alpha}= 0 \text{ for all } \bv^\alpha \in H_{00}^{\frac{1}{2}}(\partial \alpha \backslash \Gamma;\mathbb{R}^3) \},
	\end{aligned}
\end{equation}
where $\bv^\alpha$ in the duality product denotes the zero extension of $\bv^\alpha \in H_{00}^{\frac{1}{2}}(\partial \alpha \backslash \Gamma;\mathbb{R}^3)$ to $\partial \alpha$. Let $\hat{\bv}^\beta \in H^{\frac{1}{2}}(\partial \alpha;\mathbb{R}^3)$ denote a bounded extension of the restriction of $\bv^\beta \in H^1(\beta;\mathbb{R}^3)$ on $\Gamma$ to $\partial \alpha$, namely, 
\begin{equation}
\label{ineq:bounded-extension-tildev}
\| \hat{\bv}^\beta \|_{\frac{1}{2},\partial \alpha} \leq C \| \bv^\beta \|_{1, \beta}.
\end{equation}

Recall the definition of $W^\beta$ in \eqref{def:W-beta}. The new weak formulation can be established as a symmetric saddle point system as follows: Given $\bft^\alpha \in L^2(\alpha;\mathbb{R}^3)$ and $\bft^\beta \in L^2(\beta;\mathbb{R}^3)$, find $(\bsig^{\alpha}, \bu^{\alpha}, \bu^{\beta}) \in \Sigma^\alpha \times V^{\alpha}\times W^\beta$ such that
\begin{equation}
\label{pro: weak}
\begin{aligned}
    (\mathcal{C}_0^{-1} \bsig^\alpha, \btau^{\alpha})_\alpha + (\divt \btau^{\alpha}, \bu^{\alpha})_\alpha -\langle \btau^{\alpha}  \bn^{\alpha}, \hat{\bu}^{\beta} \rangle_{\partial \alpha}&= 0, \\
 (\divt \bsig^{\alpha}, \bv^{\alpha})_\alpha &= -(\bft^{\alpha}, \bv^{\alpha})_\alpha, \\
    -\langle \bsig^\alpha \bn^{\alpha}, \hat{\bv}^{\beta} \rangle_{\partial \alpha} - (\bN^{\beta}(\bu^\beta), \be^{\beta}(\bv^\beta))_\beta - (\bM^{\beta}(\bu^\beta), \bK^{\beta}(\bv^\beta))_\beta &= - (\bft^{\beta}, \bv^{\beta})_\beta,
\end{aligned}
\end{equation}
for all $(\btau^\alpha,\bv^\alpha,\bv^\beta) \in \Sigma^\alpha \times V^{\alpha}\times W^\beta$. 

The junction conditions \eqref{jun} form $H^{-\frac{1}{2}}(\partial \alpha;\mathbb{R}^3) \times H^{\frac{1}{2}}(\partial \alpha;\mathbb{R}^3)$ pairs for both the primal formulation and the mixed formulation. Specifically, the junction conditions \eqref{jun2}--\eqref{jun3} related to $\bsig^\alpha \bn^\alpha$ act as the distributional external force for the primal formulation on $\beta$, and the junction condition \eqref{jun1} related to $\bu^\beta$ acts as the natural boundary condition for the mixed formulation on $\alpha$. The weak formulation \eqref{pro: weak} can be rewritten as an equivalent saddle point system
\begin{equation}
\label{formu-asym-mixed}
\begin{aligned}
a(\bsig^\alpha,\bu^\beta; \btau^\alpha, \bv^\beta) + b(\bu^\alpha; \btau^\alpha,\bv^\beta) &= (\bft^\beta, \bv^\beta)_\beta, \\
b(\bv^\beta; \bsig^\alpha,\bu^\beta) &= -(\bft^\alpha, \bv^\alpha)_\alpha,
\end{aligned}
\end{equation}
for all $(\btau^\alpha, \bv^\alpha, \bv^\beta) \in \Sigma^\alpha \times V^{\alpha}\times W^\beta$ with 
\begin{equation}
\label{formu-asym-mixed1}
\begin{aligned}
a(\bsig^\alpha,\bu^\beta; \btau^\alpha, \bv^\beta)  &\coloneqq (\mathcal{C}_0^{-1} \bsig^\alpha, \btau^{\alpha})_\alpha + (\bN^{\beta}(\bu^\beta), \be^{\beta}(\bv^\beta))_\beta\\
&\quad + (\bM^{\beta}(\bu^\beta), \bK^{\beta}(\bv^\beta))_\beta + \langle \bsig^\alpha \bn^{\alpha}, \hat{\bv}^{\beta} \rangle_{\partial \alpha} - \langle \btau^{\alpha}  \bn^{\alpha}, \hat{\bu}^{\beta} \rangle_{\partial \alpha}, \\
b(\bu^\alpha; \btau^\alpha,\bv^\beta) &\coloneqq (\divt \btau^{\alpha}, \bu^{\alpha})_\alpha. 
\end{aligned}
\end{equation}
Although \eqref{pro: weak} and \eqref{formu-asym-mixed} are equivalent, the compact form \eqref{formu-asym-mixed} is not symmetric because of the bilinear form $a(\bullet,\bullet;\bullet,\bullet)$ in \eqref{formu-asym-mixed1}. 
This paper analyzes the well-posedness of the proposed weak formulation \eqref{pro: weak} in terms of the nonsymmetric compact form \eqref{formu-asym-mixed}. Recall the definition of $\|\bullet\|_{\divt,\alpha}$ in \eqref{def:norms}. Introduce the norm
\begin{equation}
\label{formu-asym-norm}
\interleave (\btau^\alpha,\bv^\beta)  \interleave^2 \coloneqq \| \btau^\alpha \|_{\divt,\alpha}^2 + \sum_{I=1}^2|v_I^\beta|_{1,\beta}^2 + |v_3^\beta|_{2,\beta}^2 \text{ for all } (\btau^\alpha,\bv^\beta) \in \Sigma^\alpha \times W^\beta.
\end{equation}
The well-posedness of the weak formulation \eqref{pro: weak} can be shown as follows.
\begin{theorem}
\label{thm-wellposeconti}
The weak formulation \eqref{formu-asym-mixed} is well-posed.
\end{theorem}

\begin{proof}
It follows from \eqref{ineq:bounded-extension-tildev}, the trace inequality and Poincar\'e inequality that 
$$
\begin{aligned}
\langle \bsig^\alpha \bn^\alpha, \hat{\bv}^\beta \rangle_{\partial \alpha} 
& \leq  \| \bsig^\alpha \bn^\alpha \|_{-\frac{1}{2},\partial \alpha} \| \hat{\bv}^\beta \|_{\frac{1}{2}, \partial \alpha} 
\leq C \| \bsig^\alpha \bn^\alpha \|_{-\frac{1}{2},\partial \alpha} \| \bv^\beta \|_{1, \beta}  \\
&\leq C \| \bsig^\alpha  \|_{\divt,\alpha}  \| \bv^\beta \|_{1, \beta} 
\leq C \| \bsig^\alpha  \|_{\divt,\alpha}  | \bv^\beta |_{1, \beta}.
\end{aligned}
$$
This and the definition of $\interleave \bullet \interleave$ in \eqref{formu-asym-norm} lead to the boundedness of two bilinear forms $a(\bullet,\bullet; \bullet, \bullet)$ and $b(\bullet; \bullet,\bullet)$, namely, for all $(\bsig^\alpha,\bu^\beta),(\btau^\alpha,\bv^\beta)\in \Sigma^\alpha\times W^\beta$ and $\bv^\alpha \in V^\alpha$, 
$$
\begin{aligned}
|a(\bsig^\alpha,\bu^\beta; \btau^\alpha, \bv^\beta)| \leq &  C \interleave (\bsig^\alpha,\bu^\beta)  \interleave \interleave (\btau^\alpha,\bv^\beta)  \interleave,  \\
|b(\bv^\alpha; \bsig^\alpha,\bu^\beta)| \leq& \interleave (\bsig^\alpha,\bu^\beta)  \interleave  \| \bv^\alpha \|_{0,\alpha}.
\end{aligned}
$$
Define a subspace of $\Sigma^{\alpha}$ by 
\begin{equation}
\label{def:B}
B \coloneqq \{ \btau^\alpha \in \Sigma^{\alpha} ; (\divt \btau^\alpha, \bv^\alpha)_\alpha =0 \text { for all } \bv^\alpha \in V^{\alpha}\}.     
\end{equation}
Since $\divt \Sigma^\alpha \subset V^\alpha$, it holds that $\divt \btau^\alpha=0$ for any $\btau^\alpha \in B$ and 
\begin{equation}
\label{coercive-conti}
(\mathcal{C}_0^{-1} \btau^\alpha, \btau^{\alpha})_\alpha \geq C \|\btau^\alpha\|_{\divt,\alpha} \text{ for all } \btau^\alpha \in B.    
\end{equation}
For all $(\btau^\alpha, \bv^\beta) \in \Sigma^\alpha \times V^\beta$ satisfying $(\divt \btau^\alpha, \bv^\alpha)_\alpha =0$ for all $\bv^\alpha \in V^{\alpha}$, \eqref{def:B}--\eqref{coercive-conti} and Korn's inequality with the boundary conditions \eqref{bd} lead to
$$
\begin{aligned}
a(\btau^\alpha,\bv^\beta; \btau^\alpha, \bv^\beta) &= (\mathcal{C}_0^{-1} \btau^\alpha, \btau^{\alpha})_\alpha + (\bN^{\beta}(\bv^\beta), \be^{\beta}(\bv^\beta))_\beta + (\bM^{\beta}(\bv^\beta), \bK^{\beta}(\bv^\beta))_\beta     \\
& \geq C \interleave (\btau^\alpha,\bv^\beta)  \interleave^2,
\end{aligned}
$$
which indicates that the bilinear form $a(\bullet,\bullet; \bullet, \bullet)$ is coercive.
For any $\bv^\alpha \in V^{\alpha}$, there exists a $\btau^\alpha \in \Sigma^\alpha$ such that $\divt \btau^\alpha = \bv^\alpha$ and $\|\btau^\alpha \|_{\divt,\alpha} \leq C \| \bv^{\alpha}\|_{0,\alpha}$ \cite{brenner2008mathematical}. Then it holds that 
$$
\inf_{\bv^\alpha \in V^\alpha}  \sup_{\btau^\alpha \in \Sigma^\alpha} \frac{(\divt \btau^\alpha, \bv^\alpha)_\alpha}{ \| \btau^\alpha \|_{\divt,\alpha}   \|
\bv^\alpha \|_{0,\alpha}}   \geq \frac{1}{C}
\inf_{\bv^\alpha \in V^\alpha}  \frac{\|\bv^\alpha \|_{0,\alpha}^2}{\|\bv^\alpha \|_{0,\alpha}^2} = \frac{1}{C}.
$$
Thus, the inf-sup condition of $b(\bullet; \bullet,\bullet)$ below
$$
\inf_{\bv^\alpha \in V^\alpha}  \sup_{(\btau^\alpha,\bv^\beta)\in \Sigma^\alpha \times W^\beta} \frac{b(\bv^\alpha; \btau^\alpha,\bv^\beta)}{ \interleave (\btau^\alpha,\bv^\beta)  \interleave   \|
\bv^\alpha \|_{0,\alpha}}   \geq C
$$
can obtained by taking $\bv^\beta=0$ directly. 
Brezzi's theory \cite{boffi2013mixed} completes the proof.
\end{proof}

The subsequent theorem shows that the displacement-based formulation \eqref{pro:displacement-based} is equivalent to the new weak formulation \eqref{pro: weak}.

\begin{theorem}
Given $\bft^\alpha \in L^2(\alpha;\mathbb{R}^3)$ and $\bft^\beta \in L^2(\beta;\mathbb{R}^3)$, the problem \eqref{pro:displacement-based} and the problem \eqref{pro: weak} are equivalent in the following sense: If $(\bu^\alpha, \bu^\beta)$ solves \eqref{pro:displacement-based}, then $\bsig^\alpha = \mathcal{C}_0 \bvar^\alpha(\bu^\alpha) \in \Sigma^\alpha$  and $(\bsig^\alpha, \bu^\alpha, \bu^\beta)$ solves \eqref{pro: weak}. Conversely, if $(\bsig^\alpha, \bu^\alpha, \bu^\beta)$ solves \eqref{pro: weak}, then $(\bu^\alpha, \bu^\beta)$ solves \eqref{pro:displacement-based}.
\end{theorem}

\begin{proof}
Suppose that $(\bu^\alpha, \bu^\beta)$ solves \eqref{pro:displacement-based}. Then $\bsig^\alpha \coloneqq \mathcal{C}_0 \bvar^\alpha(\bu^\alpha) \in L^2(\alpha;\mathbb{S})$ and an integration by parts leads to
\begin{equation}
\label{eq:eqivalence1}
(\mathcal{C}_0^{-1} \bsig^\alpha, \btau^\alpha)_\alpha + (\divt \btau^\alpha, \bu^\alpha)_\alpha - \langle \btau^\alpha \bn^\alpha, \bu^\alpha \rangle_{\partial \alpha}  =0 \text{ for all } \btau^\alpha \in \Sigma^\alpha.   
\end{equation}
It follows from $\bu^\alpha=\bu^\beta$ on $\Gamma$ that $\bu^\alpha - \hat{\bu}^\beta \in H_{00}^{\frac{1}{2}}(\partial \alpha \backslash \Gamma;\mathbb{R}^3)$. This and the definition of $\Sigma^\alpha$ in \eqref{defSigmaalpha} gives
\begin{equation*}
\langle \btau^\alpha \bn^\alpha, \bu^\alpha \rangle_{\partial \alpha} = \langle \btau^\alpha \bn^\alpha, \bu^\alpha \rangle_{\partial \alpha} - \langle \btau^\alpha \bn^\alpha, \bu^\alpha - \hat{\bu}^\beta\rangle_{\partial \alpha}  = \langle \btau^\alpha \bn^\alpha, \hat{\bu}^\beta \rangle_{\partial \alpha} \text{ for all } \btau^\alpha \in \Sigma^\alpha.
\end{equation*}
Substituting this into \eqref{eq:eqivalence1} proves the first row of \eqref{pro: weak}.  
The choice of $(\bv^\alpha, 0) \in \bV$ with $\bv^\alpha \in C_0^\infty(\alpha; \mathbb{R}^3)$ on $\alpha$ in \eqref{pro:displacement-based} shows $-\divt \bsig^\alpha = \bft^\alpha$. This implies that $\bsig^\alpha \in H(\divt, \alpha; \mathbb{S})$ and the second row of \eqref{pro: weak} holds. The combination of this and \eqref{pro:displacement-based} shows   
\begin{equation}
\label{eq:eqivalence2}
\langle \bsig^\alpha \bn^\alpha, \bv^\alpha \rangle_{\partial \alpha} + (\bN^{\beta}(\bu^\beta), \be^{\beta}(\bv^\beta))_\beta + (\bM^{\beta}(\bu^\beta), \bK^{\beta}(\bv^\beta))_\beta = (\bft^\beta, \bv^\beta)_\beta,
\end{equation}
for all $(\bv^\alpha, \bv^\beta)\in \bV$. Thus, $\bsig^\alpha \in \Sigma^\alpha$ since
$$\langle \bsig^\alpha \bn^\alpha, \bv^\alpha \rangle_{\partial \alpha} =0 \text{ for all }(\bv^\alpha,0)\in \bV \text{ with }\bv^\alpha \in H_{00}^{\frac{1}{2}}(\partial\alpha \backslash \Gamma;\mathbb{R}^3).$$ 
This, $\bv^\alpha=\bv^\beta$ on $\Gamma$ and \eqref{eq:eqivalence2} prove the third row of \eqref{pro: weak}.

Suppose that $(\bsig^\alpha, \bu^\alpha, \bu^\beta)$ solves \eqref{pro: weak}. 
Since $C^{\infty}_0(\alpha; \mathbb{S}) \subset \Sigma^\alpha$, it follows from the first row of \eqref{pro: weak} that
$$
(\mathcal{C}_0^{-1} \bsig^\alpha, \btau^\alpha)_\alpha + (\bu^\alpha, \divt \btau^\alpha)_\alpha = 0 \text{ for all } \btau \in C^{\infty}_0(\alpha; \mathbb{S}).
$$
This and $\bsig^\alpha \in L^2(\alpha; \mathbb{S})$ imply $\bu^\alpha \in H^1(\alpha; \mathbb{R}^3)$ and $\bvar^\alpha (\bu^\alpha) = \mathcal{C}_0^{-1} \bsig^\alpha$. 
An integration by parts leads to
\begin{equation}
\label{eq:eqivalence4}
(\mathcal{C}_0^{-1} \bsig^\alpha, \btau^\alpha)_\alpha + (\divt \btau^\alpha, \bu^\alpha)_\alpha - \langle \btau^\alpha \bn^\alpha, \bu^\alpha \rangle_{\partial \alpha}  =0 \text{ for all } \btau^\alpha \in \Sigma^\alpha.   
\end{equation}
The combination of \eqref{eq:eqivalence4} and the first row of \eqref{pro: weak} shows 
$$
\langle \btau^\alpha \bn^\alpha, \bu^\alpha \rangle_{\partial \alpha} = \langle \btau^\alpha \bn^\alpha, \hat{\bu}^\beta \rangle_{\partial \alpha} \text{ for all } \btau^\alpha \in \Sigma^\alpha,
$$
which implies $\bu^\alpha = \bu^\beta$ on $\Gamma$. Since $W^\alpha \subset V^\alpha$ and $\bsig^\alpha = \mathcal{C}_0 \bvar^\alpha(\bu^\alpha) = \bsig^\alpha(\bu^\alpha)$, an integration by parts of the second row of \eqref{pro: weak} leads to
$$
- (\bsig^\alpha(\bu^\alpha), \bvar^\alpha(\bv^\alpha))_\alpha +  \langle \bsig^\alpha(\bu^\alpha)\bn^\alpha, \bv^\alpha\rangle_{\partial \alpha} = -(\bft^\alpha, \bv^\alpha)_\alpha \text{ for all } \bv^\alpha \in W^\alpha. 
$$
Noting that $\bv^\alpha = \bv^\beta$ on $\Gamma$ and $\bsig^\alpha \in \Sigma^\alpha$, the combination of this and the third row of \eqref{pro: weak} gives \eqref{pro:displacement-based}, which completes the proof.
\end{proof}


\section{Finite element methods}

This section provides conforming and nonconforming finite element methods for the new weak formulation  \eqref{pro: weak} of the coupled body-plate model. Under some assumptions, the discrete stability is established and error estimates are derived for them, respectively. Specific examples are presented for illustration. 

Let $\alpha$ and $\beta$ be subdivided by families of shape-regular tetrahedral grids $\mathcal{T}_h^\alpha$ and triangular grids $\mathcal{T}_h^\beta$, respectively. Note that the non-matching meshes on the interface $\Gamma$ are allowed.  Let $h_{K^\alpha}$ and $h_\alpha$ be the diameter of $K^\alpha$ and the maximum of the diameters of all the elements $K^\alpha \in \mathcal{T}_h^\alpha$, respectively. For convenience, $\bn^\alpha$ also denotes the unit outward normal vector of $K^\alpha$. The symbols $h_{K^\beta}$, $h_\beta$ and $\bn^\beta$ are defined similarly for $K^\beta \in \mathcal{T}_h^\beta$.  Denote by $\mathcal{F}(K^\alpha)$ the set of all faces of $K^\alpha$ and $h_{F}$ the diameter of $F \in \mathcal{F}(K^\alpha)$. Let $\mathcal{F}_\Gamma^\alpha$ denote the restriction of $\mathcal{T}_h^\alpha$ on $\Gamma$, which forms a triangulation of $\Gamma$ as follows:
\begin{equation}
\label{def:F-gamma}
\mathcal{F}_\Gamma^\alpha \coloneqq \left\{ F; F \subset \Gamma,  F \in \mathcal{F}(K^\alpha) \text{ for any } K^\alpha \in \mathcal{T}_h^\alpha \right\}.
\end{equation}

\subsection{Conforming finite element method}
\label{sub:conformingelement}

This subsection considers the conforming finite element method based on the weak formulation \eqref{pro: weak}. Choose the conforming discrete spaces as follows:
\begin{equation}
\Sigma^\alpha_h \subset \Sigma^\alpha, \,  V_h^\alpha \subset V^{\alpha}, \, W_{1h}^{\beta} \subset W_1^\beta, \, W_{2h}^{\beta} \subset W_2^\beta.
\end{equation}
In accordance with \eqref{def:W-beta}, let $W_h^\beta \coloneqq W_{1h}^\beta \times W_{2h}^\beta$.
Then the conforming finite element method reads:
Given $\bft^\alpha \in L^2(\alpha;\mathbb{R}^3)$ and $\bft^\beta \in L^2(\beta;\mathbb{R}^3)$, find $(\bsig_h^{\alpha}, \bu_h^{\alpha}, \bu_h^{\beta}) \in \Sigma_h^\alpha \times V_h^{\alpha}\times W_h^\beta$ such that 
\begin{equation}
\label{pro: weak-discrete-c}
\begin{aligned}
a(\bsig_h^\alpha,\bu_h^\beta; \btau_h^\alpha, \bv_h^\beta) + b(\bu_h^\alpha; \btau_h^\alpha,\bv_h^\beta) &= (\bft^\beta, \bv_h^\beta)_\beta,\\
b(\bv_h^\alpha; \bsig_h^\alpha,\bu_h^\beta) &= -(\bft^\alpha, \bv_h^\alpha)_\alpha,
\end{aligned}    
\end{equation}
for all $(\btau_h^\alpha, \bv_h^\alpha,\bv_h^\beta) \in \Sigma_h^\alpha \times V_h^{\alpha}\times W_h^\beta$. 
Since $\btau_h^\alpha \bn^\alpha|_{\partial \alpha \backslash \Gamma}=0$ and $\btau_h^\alpha \bn^\alpha \in L^2(\Gamma;\mathbb{R}^3)$, the duality product $\langle\btau_h^\alpha \bn^\alpha, \hat{\bu}_h^\beta\rangle_{\partial \alpha}$ in $a(\bsig_h^\alpha,\bu_h^\beta; \btau_h^\alpha, \bv_h^\beta)$ of \eqref{pro: weak-discrete-c} becomes the $L^2$ inner product $(\btau_h^\alpha \bn^\alpha, \bu_h^\beta)_{\Gamma}$. Similarly, $\langle \bsig_h^\alpha \bn^\alpha, \hat{\bv}_h^\beta\rangle_{\partial \alpha}$ becomes $(\bsig_h^\alpha \bn^\alpha, \bv_h^\beta )_{\Gamma}$.

Note that the classical conforming element spaces in $H^1$ and $H^2$ are employed for the discrete spaces $W_{1h}^\beta$ and $W_{2h}^\beta$ in $\beta$, respectively, while the conforming mixed element space in $H(\divt;\mathbb{S})\times L^2$ utilized for the discrete space $\Sigma_h^\alpha \times V_h^\alpha$ in $\alpha$. 
Similar to the weak formulation in \eqref{formu-asym-mixed}, the discrete formulation in \eqref{pro: weak-discrete-c} can be rewritten as a symmetric saddle point system, allowing straightforward application of a wide range of standard solvers.

To show the discrete stability of the conforming finite element method, assume that the conforming finite element pair $\Sigma_h^\alpha \times V_h^\alpha$ on $\alpha$ is stable. Namely, the coercivity holds as follows:
\begin{equation}
\label{ass: A1}
(\mathcal{C}_0^{-1} \btau_h^\alpha, \btau_h^\alpha)_{\alpha} \geq C \| \btau_h^\alpha \|_{\divt,\alpha} \text{ for all } \btau_h^\alpha \in B_h^c,    
\end{equation}
with 
$$
B_h^c \coloneqq \{ \btau_h^\alpha \in \Sigma_h^\alpha; (\divt \btau_h^\alpha, \bv_h^\alpha)_\alpha=0 \text{ for all }  \bv_h^\alpha \in V_h^\alpha \},
$$
and the inf-sup condition holds as follows:
\begin{equation}
\label{ass: A}
\inf_{\bv_h^\alpha \in V_h^\alpha} \sup_{\btau_h^\alpha \in \Sigma_h^\alpha} \frac{(\divt \btau_h^\alpha, \bv_h^\alpha)_\alpha}{\| \bv_h^\alpha\|_{0,\alpha} \| \btau_h^\alpha\|_{\divt,\alpha}} \geq C.    
\end{equation}

\begin{theorem}
 \label{thm-wellposediscrete}
    Problem \eqref{pro: weak-discrete-c} is well-posed under assumptions \eqref{ass: A1}--\eqref{ass: A}.
\end{theorem}
\begin{proof}
    It is easy to verify that $a(\bullet,\bullet; \bullet, \bullet)$ is a nonnegative bilinear form. It remains to show that Brezzi's conditions \cite{boffi2013mixed} hold for \eqref{pro: weak-discrete-c}, which can be derived from the continuous counterpart in Theorem \ref{thm-wellposeconti} and assumptions \eqref{ass: A1}--\eqref{ass: A}. The details are omitted for brevity.
\end{proof}
  
Following the standard procedures in \cite{boffi2013mixed}, the well-posedness of problem \eqref{pro: weak-discrete-c} allows the following error estimate.
\begin{theorem} 
    \label{theoremConvergence}
    Let $(\bsig^\alpha,\bu^\alpha,\bu^\beta) \in \Sigma^\alpha \times V^\alpha \times W^\beta$ be the solution of problem \eqref{pro: weak} and $(\bsig_h^\alpha,\bu_h^\alpha,\bu_h^\beta) \in \Sigma_h^\alpha \times V_h^\alpha \times W_h^\beta$ be the solution of \eqref{pro: weak-discrete-c}. 
    Then 
    \begin{equation*}
        \begin{aligned}
            & \interleave (\bsig^\alpha-\bsig_h^\alpha, \bu^\beta-\bu_h^\beta) \interleave + \| \bu^\alpha -\bu_h^\alpha \|_{0,\alpha}   \\
            &\quad  \quad \quad \leq C \inf_{( \btau_h^\alpha,\bv_h^\alpha,\bv_h^\beta)\in \Sigma_h^\alpha \times V^\alpha_h \times W_h^\beta   } \left( \interleave (\bsig^\alpha-\btau_h^\alpha, \bu^\beta-\bv_h^\beta) \interleave + \| \bu^\alpha -\bv_h^\alpha \|_{0,\alpha}\right).
    \end{aligned}
    \end{equation*}
\end{theorem}

As an example of a conforming finite element method, choose the vectorial $P_4$ Lagrange finite element for the plane elasticity model and the Argyris finite element for the plate bending model. To achieve the optimal convergence rate, choose a conforming $P_4$-$P_3$ mixed $H(\divt;\mathbb{S})\times L^2$ finite element in \cite{Hu2015,HuZhang2015}, which satisfies assumptions \eqref{ass: A1}--\eqref{ass: A}. Assume that the exact solution is smooth enough, it follows from Theorem \ref{theoremConvergence} and the usual interpolation estimates that 
\begin{equation*}
\interleave (\bsig^\alpha-\bsig_h^\alpha, \bu^\beta-\bu_h^\beta) \interleave + \| \bu^\alpha -\bu_h^\alpha \|_{0,\alpha} \leq C (h_\alpha^4+h_\beta^4).
\end{equation*}
 
\subsection{Nonconforming finite element method}
Recall that $V^{\alpha} = L^2(\alpha;\mathbb{R}^3)$ and $W_{1}^\beta=H_0^1(\beta;\mathbb{R}^2)$. The construction of conforming discrete
 spaces for $V^{\alpha}$ and $W_1^\beta$ is straightforward, whereas constructing conforming finite elements for $\Sigma^\alpha$ in \eqref{defSigmaalpha} and $W_2^\beta=H_0^2(\beta)$ is considerably more involved. Thus, this subsection focuses on the nonconforming finite element method based on the weak formulation \eqref{pro: weak}. 
 Choose discrete spaces as follows:
\begin{equation}
\Sigma^\alpha_h \not \subset \Sigma^\alpha, \, V_h^\alpha \subset V^\alpha, \, W_{1h}^\beta \subset W_1^\beta, \, W_{2h}^{\beta} \not \subset W_2^\beta.
\end{equation}
In accordance with \eqref{def:W-beta}, let $W_h^\beta \coloneqq W_{1h}^\beta \times W_{2h}^\beta$.

Let $\divt_h$ be the piecewise divergence operator defined on $\mathcal{T}_h^\alpha$. Define $\bK_h^\beta(\bv^\beta)$ and $\bM_h^\beta(\bv^\beta)$ as the piecewise counterparts of $\bK^\beta(\bv^\beta)$ and $\bM^\beta(\bv^\beta)$ for piecewise functions $\bv^\beta$ on $\mathcal{T}_h^\beta$, respectively, as introduced in \eqref{def:KandM}. To define the mesh-dependent norm below, introduce a subspace $\tilde{\Sigma}^\alpha \subset \Sigma^\alpha$ as follows:
\begin{equation}
\label{def:tSigma}
\tilde{\Sigma}^\alpha \coloneqq  \Sigma^\alpha \cap H^s(\alpha;\mathbb{S}) \text{ with } s > \frac{1}{2}. 
\end{equation}
Recall the definition of $\mathcal{F}_\Gamma^\alpha$ in \eqref{def:F-gamma}. For all $\btau_h^\alpha \in  \tilde{\Sigma}^\alpha + \Sigma_h^\alpha$, the mesh-dependent norm reads as
\begin{equation}
\label{def:divh}
\| \btau_h^\alpha\|_{\divt,h}^2 \coloneqq \| \btau_h^\alpha\|_{0,\alpha}^2 + \| \divt_h \btau_h^\alpha \|_{0,\alpha}^2+ \sum_{F\in \mathcal{F}_\Gamma^\alpha} h_F \| \btau_h^\alpha \bn^\alpha\|_{0,F}^2.     
\end{equation}
For all $v_{3h}^\beta \in W_{2h}^\beta$ and integer $m \geq 0$, define the discrete norm and semi-norm
$$
\|v_{3h}^\beta\|_{m,h}^2 \coloneqq  \sum_{K^\beta \in \mathcal{T}_h^\beta} \|v_{3h}^\beta\|_{m,K^\beta}^2, \,
|v_{3h}^\beta|_{m,h}^2 \coloneqq  \sum_{K^\beta \in \mathcal{T}_h^\beta} |v_{3h}^\beta|_{m,K^\beta}^2.
$$ 
Let $\Pi_3^\beta$ be an interpolation operator from $W_{2h}^\beta$ onto $H^1(\beta;\mathbb{R}^3)$ and define $\Pi^\beta \bv_h^\beta = ( v_{1h}^\beta, v_{2h}^\beta, \Pi_{3}^\beta v_{3h}^\beta)$ for all $\bv_h^\beta={( v_{1h}^\beta ,v_{2h}^\beta, v_{3h}^\beta)} \in W_h^\beta$. 
Define
$$
\begin{aligned}
a_h(\bsig_h^\alpha,\bu_h^\beta; \btau_h^\alpha, \bv_h^\beta)  &\coloneqq (\mathcal{C}_0^{-1} \bsig_h^\alpha, \btau_h^{\alpha})_\alpha + (\bN^{\beta}(\bu_h^\beta), \be^{\beta}(\bv_h^\beta))_\beta+ 
(\bM_h^\beta(\bu_{h}^\beta), \bK_h^\beta(\bv_{h}^\beta))_\beta  \\
& \quad + (\bsig_h^\alpha \bn^{\alpha},\Pi^\beta \bv_h^{\beta} )_{\Gamma} - (\btau_h^{\alpha}  \bn^{\alpha}, \Pi^\beta \bu_h^\beta )_{\Gamma}, \\
b_h(\bu_h^\alpha; \btau_h^\alpha,\bv_h^\beta) &\coloneqq (\divt_h \btau_h^{\alpha}, \bu_h^{\alpha})_\alpha,
\end{aligned}
$$
for all $(\bsig_h^\alpha,\bu_h^\beta),(\btau_h^\alpha,\bv_h^\beta) \in (\tilde{\Sigma}^\alpha+\Sigma_h^\alpha) \times (W^\beta + W_h^\beta)$ and $\bu_h^\alpha \in V^\alpha$.
The nonconforming finite element method reads:
Given $\bft^\alpha \in L^2(\alpha;\mathbb{R}^3)$ and $\bft^\beta \in L^2(\beta;\mathbb{R}^3)$, find $(\bsig_h^{\alpha}, \bu_h^{\alpha}, \bu_h^{\beta}) \in \Sigma_h^\alpha \times V_h^{\alpha}\times W_h^\beta$ such that 
\begin{equation}
\label{pro: weak-discrete-nc}
\begin{aligned}
a_h(\bsig_h^\alpha,\bu_h^\beta; \btau_h^\alpha, \bv_h^\beta) + b_h(\bu_h^\alpha; \btau_h^\alpha,\bv_h^\beta) &= (\bft^\beta, \bv_h^\beta)_\beta, \\
b_h(\bv_h^\alpha; \bsig_h^\alpha,\bu_h^\beta) &= -(\bft^\alpha, \bv_h^\alpha)_\alpha, 
\end{aligned}    
\end{equation}
for all $(\btau_h^\alpha, \bv_h^\alpha,\bv_h^\beta) \in \Sigma_h^\alpha \times V_h^{\alpha}\times W_h^\beta$. 

\begin{remark}
The operator $\Pi^\beta$ is pivotal for controlling the degree of nonconformity. Such control is essential for the boundedness of $a_h(\bullet,\bullet;\bullet,\bullet)$, notably facilitating the analysis without requiring specific mesh conditions along the interface. Similar interpolation operators are introduced in analyzing nonconforming finite element methods, see e.g., \cite{veeser2019quasi, ming2006morley, shi1990error}.
\end{remark}

The following assumptions are proposed for the discrete stability and error analysis of \eqref{pro: weak-discrete-nc}. 
\begin{enumerate}
    \item[(A1)] The following $\interleave \bullet \interleave_{h}$ is a norm on $(\tilde{\Sigma}^\alpha+\Sigma^\alpha_h) \times (W^\beta+W_h^\beta)$
\begin{equation*}
\interleave (\btau_h^\alpha,\bv_h^\beta)  \interleave_{h}^2 \coloneqq \| \btau_h^\alpha\|_{\divt,h}^2  + \sum_{I=1}^2|v_{Ih}^\beta|_{1,\beta}^2 + |v_{3h}^\beta|_{2,h}^2.
\end{equation*}.
\item[(A2)] The bilinear forms $a_h(\bullet, \bullet;\bullet,\bullet)$ and $b_h(\bullet;\bullet, \bullet)$ are bounded with respect to the norms $\interleave \bullet \interleave_{h}$ and $\| \bullet \|_{0,\alpha}$, namely,
$$
\begin{aligned}
|a_h(\bsig_h^\alpha,\bu_h^\beta; \btau_h^\alpha, \bv_h^\beta)| &\leq C \interleave(\bsig_h^\alpha,\bu_h^\beta)\interleave_{h} \interleave (\btau_h^\alpha, \bv_h^\beta) \interleave_{h}, \\
|b_h(\bv_h^\alpha; \bsig_h^\alpha,\bu_h^\beta)| &\leq C  \interleave(\bsig_h^\alpha,\bu_h^\beta)\interleave_{h}\| \bv_h^\alpha\|_{0,\alpha},
\end{aligned}
$$
for all $(\bsig_h^\alpha,\bu_h^\beta) \in (\tilde{\Sigma}^\alpha+\Sigma_h^\alpha)\times (W^\beta+W_h^\beta)$, $(\btau_h^\alpha, \bv_h^\beta) \in \Sigma_h^\alpha \times W_h^\beta$ and $\bv_h^\alpha \in V_h^\alpha$.
\item[(A3)] The mixed finite element pair $\Sigma_h^\alpha \times V_h^\alpha$ on $\alpha$ is stable. Namely, the coercivity holds
$$
(\mathcal{C}_0^{-1} \btau_h^\alpha, \btau_h^\alpha)_{\alpha} \geq C \| \btau^\alpha \|_{\divt,h} \text{ for all } \btau^\alpha \in B_h^{nc}     
$$
with 
$$
B_h^{nc} \coloneqq \{ \btau_h^\alpha \in \Sigma_h^\alpha; (\divt_h \btau_h^\alpha, \bv_h^\alpha)_\alpha=0 \text{ for all }  \bv_h^\alpha \in V_h^\alpha \},
$$
and the inf-sup condition holds
$$
\inf_{\bv_h^\alpha \in V_h^\alpha} \sup_{\btau_h^\alpha \in \Sigma_h^\alpha} \frac{(\divt \btau_h^\alpha, \bv_h^\alpha)_\alpha}{\| \bv_h^\alpha\|_{0,\alpha} \| \btau_h^\alpha\|_{\divt,h}} \geq C.    
$$
\item[(A4)]The approximation property holds 
$$
\inf_{(\btau_h^\alpha, \bv_h^\beta)\in \Sigma_h^\alpha \times W_h^\beta} \interleave ( \bsig^\alpha - \btau_h^\alpha, \bu^\beta- \bv_h^\beta) \interleave_{h} + \inf_{\bv_h^\alpha \in V_h^\alpha} \| \bu^\alpha - \bv_h^\alpha \|_{0,\alpha} \leq C (h_\alpha^k + h_\beta^k)
$$
with a positive integer $k$.
\item[(A5)]The consistent error estimate holds
$$
\sup_{(\btau_h^\alpha, \bv_h^\beta) \in \Sigma_h^\alpha \times W_h^\beta} \frac{ |T_0|}{\interleave (\btau_h^\alpha, \bv_h^\beta) \interleave_{h}} \leq C (h_\alpha^k + h_\beta^k)
$$
with a positive integer $k$ and 
\begin{equation}
\label{def:T0}
T_0 \coloneqq a_h(\bsig^\alpha, \bu^\beta; \btau_h^\alpha, \bv_h^\beta) + b_h(\bu^\alpha; \btau_h^\alpha, \bv_h^\beta) - (\bft^\beta, \bv_h^\beta)_\beta.
\end{equation}
\end{enumerate}

\begin{remark}
\label{remark:ah}
To analyze the consistent error, the subspace $\tilde{\Sigma}^\alpha \subset \Sigma^\alpha$ is introduced, and the exact solution $\bsig^\alpha$ is assumed to belong to $\tilde{\Sigma}^\alpha$. This assumption ensures that the extended definition of $a_h(\bullet,\bullet;\bullet,\bullet)$ is applicable to the exact solution, thereby satisfying the nonconforming finite element method \eqref{pro: weak-discrete-nc}.
\end{remark}

\begin{remark}
\label{remark:A2}
The well-posedness of the nonconforming finite element method \eqref{pro: weak-discrete-nc} can be derived by Brezzi's conditions, which requires the boundedness of the bilinear forms $a_h(\bsig_h^\alpha,\bu_h^\beta;\btau_h^\alpha, \bv_h^\beta)$ and $b_h(\bv_h^\alpha;\bsig_h^\alpha,\bu_h^\beta)$ for all $(\bsig_h^\alpha,\bu_h^\beta), (\btau_h^\alpha, \bv_h^\beta) \in \Sigma_h^\alpha \times W_h^\beta$ and $\bv_h^\alpha \in V_h^\alpha$. Assumption $\mathrm{(A2)}$ requires that the boundedness holds for all $(\bsig_h^\alpha,\bu_h^\beta) \in (\tilde{\Sigma}^\alpha+\Sigma_h^\alpha)\times (W^\beta+W_h^\beta)$ to analyze the error estimate.
\end{remark}

\begin{theorem}
\label{thm:conti-nonconforming-error}
Under Assumptions $\mathrm{(A1)}$--$\mathrm{(A3)}$, problem \eqref{pro: weak-discrete-nc} is well-posed, i.e., there exists a unique solution $(\bsig_h^\alpha, \bu_h^\alpha, \bu_h^\beta) \in \Sigma_h^\alpha \times V_h^\alpha \times W_h^\beta$ such that
$$
\interleave (\bsig_h^\alpha, \bu_h^\beta) \interleave_{h} + \| \bu_h^\alpha \|_{0,\alpha} \leq C ( \| \bft^\alpha\|_{0,\alpha} + \| \bft^\beta\|_{0,\beta}).
$$
Moreover, if Assumptions $\mathrm{(A4)}$--$\mathrm{(A5)}$ hold, then the optimal convergence holds
$$
\interleave ( \bsig^\alpha - \bsig_h^\alpha, \bu^\beta- \bu_h^\beta) \interleave_{h} + \| \bu^\alpha - \bu_h^\alpha \|_{0,\alpha} \leq C (h_\alpha^k+h_\beta^k).
$$
\end{theorem}

\begin{proof}
Assumption $\mathrm{(A2)}$ gives the boundedness of $a_h(\bullet,\bullet;\bullet,\bullet)$ and $b_h(\bullet;\bullet,\bullet)$, namely,
\begin{equation*}
\label{thm:4.53}
\begin{aligned}
|a_h(\bsig_h^\alpha,\bu_h^\beta; \btau_h^\alpha, \bv_h^\beta)| &\leq C \interleave(\bsig_h^\alpha,\bu_h^\beta)\interleave_{h} \interleave (\btau_h^\alpha, \bv_h^\beta) \interleave_{h}, \\
|b_h(\bv_h^\alpha; \bsig_h^\alpha,\bu_h^\beta)| &\leq C  \interleave(\bsig_h^\alpha,\bu_h^\beta)\interleave_{h}\| \bv_h^\alpha\|_{0,\alpha},
\end{aligned}
\end{equation*}
for all $(\bsig_h^\alpha,\bu_h^\beta), (\btau_h^\alpha, \bv_h^\beta) \in \Sigma_h^\alpha \times W_h^\beta$ and $\bv_h^\alpha \in V_h^\alpha$. 
By Assumption $\mathrm{(A3)}$, the coercivity holds as follows:
\begin{equation}
\label{thm:4.51}
(\mathcal{C}_0^{-1} \btau_h^\alpha, \btau_h^\alpha)_\alpha \geq C \| \btau_h^\alpha \|_{\divt,h} \text{ for all } \btau_h^\alpha \in B_h^{nc}.    
\end{equation}
Since $\interleave \bullet \interleave_{h}$ is a norm by Assumption $\mathrm{(A1)}$,  Korn's inequality and \eqref{thm:4.51} show that the bilinear form $a_h(\bullet,\bullet;\bullet,\bullet)$ is coercive, namely, 
\begin{equation*}
\label{thm:4.52}
\begin{aligned}
a_h(\btau_h^\alpha, \bv_h^\beta; \btau_h^\alpha, \bv_h^\beta) &=   (\mathcal{C}_0^{-1} \btau_h^\alpha, \btau_h^\alpha)_\alpha + (\bsig^\beta(\bv_h^\beta), \bvar^\beta(\bv_h^\beta))_\beta + (\bM_h^\beta(\bv_{h}^\beta), \bK_h^\beta(\bv_{h}^\beta))_\beta \\
& \geq  C \interleave (\btau_h^\alpha, \bv_h^\beta) \interleave_{h}^2.
\end{aligned}    
\end{equation*}
for all $(\btau_h^\alpha, \bv_h^\beta) \in \Sigma_h^\alpha \times V_h^\beta$ satisfying $\btau_h^\alpha \in B_h^{nc}$.
The inf-sup condition of $b_h(\bullet;\bullet,\bullet)$ follows from the inf-sup condition in Assumption $\mathrm{(A3)}$ by taking $\bv_h^\beta=0$ as follows:
\begin{equation*}
\begin{aligned}
\inf_{\bv_h^\alpha \in V_h^\alpha} \sup_{(\btau_h^\alpha,\bv_h^\beta) \in \Sigma_h^\alpha\times W_h^\beta} \frac{b_h(\bv_h^\alpha; \btau_h^\alpha,\bv_h^\beta)}{\| \bv_h^\alpha\|_{0,\alpha} \| \btau_h^\alpha\|_{\divt,h}} = 
\inf_{\bv_h^\alpha \in V_h^\alpha} \sup_{\btau_h^\alpha \in \Sigma_h^\alpha} \frac{(\divt \btau_h^\alpha, \bv_h^\alpha)_\alpha}{\| \bv_h^\alpha\|_{0,\alpha} \| \btau_h^\alpha\|_{\divt,h}} \geq C. 
\end{aligned}
\end{equation*}
Brezzi's conditions \cite{boffi2013mixed} show that problem \eqref{pro: weak-discrete-nc} is well-posed. 

Let $(\bar{\bsig}^\alpha, \bar{\bu}^\alpha, \bar{\bu}^\beta) \in \Sigma_h^\alpha \times V_h^\alpha \times W_{h}^\beta$ denote an arbitrary discrete function. The triangle inequality leads to
\begin{equation*}
\begin{aligned}
& \interleave ( \bsig^\alpha - \bsig_h^\alpha, \bu^\beta- \bu_h^\beta) \interleave_{h} + \| \bu^\alpha - \bu_h^\alpha \|_{0,\alpha} \\
&\quad \leq  \left(  \interleave ( \bsig^\alpha - \bar{\bsig}^\alpha, \bu^\beta- \bar{\bu}^\beta) \interleave_{h} + \| \bu^\alpha - \bar{\bu}^\alpha \|_{0,\alpha} \right)\\
&\quad + \left( \interleave ( \bar{\bsig}^\alpha - \bsig_h^\alpha, \bar{\bu}^\beta- \bu_h^\beta) \interleave_{h} + \| \bar{\bu}^\alpha - \bu_h^\alpha \|_{0,\alpha} \right) \coloneqq E_1 + E_2.
\end{aligned}      
\end{equation*}  
By Assumption $\mathrm{(A4)}$, the approximation error $E_1$ can be estimated by
\begin{equation}
\label{thm:4.502}
|E_1| \leq C(h_\alpha^k + h_\beta^k).
\end{equation}
for some positive integer $k$. To estimate $E_2$, define 
$$
\begin{aligned}
T_1 & \coloneqq a_h( \bar{\bsig}^\alpha - \bsig_h^\alpha, \bar{\bu}^\beta-\bu_h^\beta; \btau_h^\alpha, \bv_h^\beta)+  b_h(\bar{\bu}^\alpha-\bu_h^\alpha; \btau_h^\alpha, \bv_h^\beta), \\
T_2 & \coloneqq b_h(\bv_h^\alpha; \bar{\bsig}^\alpha - \bsig_h^\alpha, \bar{\bu}^\beta-\bu_h^\beta).
\end{aligned}    
$$
This, Brezzi's theory \cite{boffi2013mixed} and Assumption $\mathrm{(A3)}$ yield the stability result
\begin{equation}
\label{eq:4.503}
\begin{aligned}
 E_2 &\leq \sup_{(\btau_h^\alpha, \bv_h^\beta)\in \Sigma_h^\alpha \times W_h^\beta, \bv_h^\alpha \in V_h^\alpha}  \frac{|T_1 + T_2| }{ \interleave (\btau_h^\alpha, \bv_h^\beta) \interleave_{h} + \| \bv_h^\alpha\|_{0,\alpha}}\\
& \leq \sup_{(\btau_h^\alpha, \bv_h^\beta)\in \Sigma_h^\alpha \times W_h^\beta}  \frac{|T_1| }{ \interleave (\btau_h^\alpha, \bv_h^\beta) \interleave_{h} } + \sup_{\bv_h^\alpha \in V_h^\alpha}  \frac{|T_2| }{\| \bv_h^\alpha\|_{0,\alpha}}.
\end{aligned}    
\end{equation}
Since $\bv_h^\alpha \in V_h^\alpha \subset V^\alpha$, the combination of \eqref{pro: weak} and \eqref{pro: weak-discrete-nc} leads to
$$
\begin{aligned}
T_2 &= ( \divt_h (\bar{\bsig}^\alpha - \bsig_h^\alpha), \bv_h^\alpha)_\alpha = ( \divt_h (\bar{\bsig}^\alpha - \bsig^\alpha), \bv_h^\alpha)_\alpha.  
\end{aligned}
$$
Assumptions  $\mathrm{(A2)}$ and $\mathrm{(A4)}$ show
\begin{equation}
\label{eq:4.505}
\sup_{\bv_h^\alpha \in V_h^\alpha}  \frac{|T_2| }{\| \bv_h^\alpha\|_{0,\alpha}} \leq \interleave (\bsig^\alpha - \bar{\bsig}^\alpha, \bu^\beta - \bar{\bu}^\beta)\interleave_{h} \leq C(h_\alpha^k + h_\beta^k).  
\end{equation}
By \eqref{pro: weak-discrete-nc}, it holds that 
$$
\begin{aligned}
T_1  = a_h(\bar{\bsig}^\alpha, \bar{\bu}^\beta; \btau_h^\alpha, \bv_h^\beta) + b_h(\bar{\bu}^\alpha; \btau_h^\alpha, \bv_h^\beta) - (\bft^\beta, \bv_h^\beta)_\beta. \\
\end{aligned}
$$
Recall the definition of $T_0$ in \eqref{def:T0}. The triangle inequality leads to
$$
\sup_{(\btau_h^\alpha, \bv_h^\beta)\in \Sigma_h^\alpha \times W_h^\beta}  \frac{|T_1| }{ \interleave (\btau_h^\alpha, \bv_h^\beta) \interleave_{h} }  \leq \sup_{(\btau_h^\alpha, \bv_h^\beta)\in \Sigma_h^\alpha \times W_h^\beta} \left( \frac{|T_1- T_0| }{ \interleave (\btau_h^\alpha, \bv_h^\beta) \interleave_{h} } + \frac{|T_0| }{ \interleave (\btau_h^\alpha, \bv_h^\beta) \interleave_{h} }\right). 
$$
The combination of Assumptions $\mathrm{(A2)}$, $\mathrm{(A4)}$ and $\mathrm{(A5)}$ gives 
\begin{equation}
\label{eq:4.504}
\sup_{(\btau_h^\alpha, \bv_h^\beta)\in \Sigma_h^\alpha \times W_h^\beta}  \frac{|T_1| }{ \interleave (\btau_h^\alpha, \bv_h^\beta) \interleave_{h} } \leq C(h_\alpha^k + h_\beta^k).
\end{equation}
Substituting \eqref{eq:4.505} and \eqref{eq:4.504} into \eqref{eq:4.503} gives the estimate of $E_2$, which combines with \eqref{thm:4.502} completes the proof.
\end{proof}

\subsection{An example}
\label{sub:nonconformingelement}
This subsection presents an example of the nonconforming finite element method satisfying Assumptions $\mathrm{(A1)}$--$\mathrm{(A5)}$. 

Let $x_i \, (1 \leq i \leq 4)$ denote the vertices of tetrahedron $K^\alpha$ and $\boldsymbol{t}_{i,j} = x_j - x_i \,(i \neq j)$ the tangential vector of edge $x_i x_j$. Let $\lambda_i \, (1 \leq i \leq 4)$ denote the barycentric coordinates with respect to $x_i$.
Let $\Pi^\alpha_{SZ}$ be the Scott-Zhang interpolation operator \cite{scott1990finite} onto  continuous piecewise linear functions on $\mathcal{T}_h^\alpha$. 
Let $S_{K^\alpha}$ denote the set of elements in $\mathcal{T}_h^\alpha$ which have nonempty intersection with the closure of $K^\alpha$.
Given $K^\alpha \in \mathcal{T}_h^\alpha$, it holds that
\begin{equation}
\label{SZ}
\| \bv^\alpha - \Pi_{SZ}^\alpha \bv^\alpha \|_{0,K^\alpha} +h_{K^\alpha} \| \bv^\alpha - \Pi_{SZ}^\alpha \bv^\alpha \|_{1,K^\alpha} \leq C h_{K^\alpha}^{s} \|\bv^\alpha\|_{s,S_{K^\alpha}},
\end{equation}
for all $\bv^\alpha \in H^s(\alpha;\mathbb{R}^3)$ with $s=1,2$. Define
\begin{equation}
\label{def:interface-product}
(\bullet, \bullet)_{\partial \mathcal{T}_h^\alpha} \coloneqq \sum_{K^\alpha \in \mathcal{T}_h^\alpha}(\bullet, \bullet)_{\partial K^\alpha}.    
\end{equation} 

Choose the nonconforming $H(\divt,\alpha;\mathbb{S})\times L^2(\alpha;\mathbb{R}^3)$ mixed finite element in \cite{HuMa2018} with $k=1$ for the three dimensional body. The discrete displacement space is the discontinuous piecewise linear space
\[
V_{h}^{\alpha} \coloneqq \{ \bv_h^\alpha \in L^2(\alpha; \mathbb{R}^3); \bv_h^\alpha|_{K^\alpha} \in P_1(K^\alpha; \mathbb{R}^3) \text{ for all } K^\alpha \in \mathcal{T}_h^\alpha \}.
\]
For $1 \leq i < j \leq 4$ and $\{\ell, m\} = \{1, 2, 3, 4\} \backslash \{i, j\}$, define
\[
P_{ij} \coloneqq P_1(K^\alpha) + (\lambda_i - \lambda_j) \text{span}\{\lambda_\ell, \lambda_m\} + \lambda_i \lambda_j P_0(K^\alpha).
\]
The discrete stress space reads
\begin{equation}
\label{def-sigmah}
\begin{aligned}
\Sigma_h^{\alpha} \coloneqq &  \{ \btau_h^\alpha \in L^2(\alpha;\mathbb{S}); \btau_h^\alpha = \sum_{1 \leq i < j \leq 4} p_{ij} \boldsymbol{t}_{ij} \boldsymbol{t}_{ij}^T,  p_{ij} \in P_{ij} \text{ for all } K^\alpha \in \mathcal{T}_h^\alpha, \\
& \text{moments of } \btau_h^\alpha \bn{^\alpha} \text{ up to degree } 1
\text{ are continuous across}\\
& \text{internal faces and equal zero on boundary faces on } \partial \alpha \backslash \Gamma \}.    
\end{aligned}     
\end{equation}
Choose the vectorial continuous piecewise linear Lagrange finite element with zero boundary conditions for the plane elasticity model defined as follows:
$$
\begin{aligned}
W_{1h}^\beta \coloneqq \{& (v_{1h}^\beta,v_{2h}^\beta) \in H^1(\beta;\mathbb{R}^2); (v_{1h}^\beta,v_{2h}^\beta)|_{K^\beta} \in P_1(K^\beta;\mathbb{R}^2), (v_{1h}^\beta,v_{2h}^\beta)\text{ vanishes}\\
&\text{at the vertices at } \partial \beta \}.
\end{aligned}
$$
Choose the Morley finite element space for the plate bending model defined as follows:
$$
\begin{aligned}
W_{2h}^\beta \coloneqq \{
&v_{3h}^\beta \in L^2(\beta); v_{3h}^\beta|_{K^\beta} \in P_2(K^\beta) \text{ for all } K^\beta \in \mathcal{T}_h^\beta, v_{3h}^\beta \text{ is continuous at vertices} \\
&\text{and vanishes at the vertices at } \partial \beta, \partial_{n^{\beta}} v_{3h}^\beta \text{ is continuous at the midpoint} \\
& \text{of each internal edge and vanishes at the midpoint of each edge on } \partial \beta \}.
\end{aligned}
$$
Let $\Pi_3^\beta$ be the nodal interpolation defined on $W_2^\beta+W_{2h}^\beta$ to continuous piecewise linear functions on $\mathcal{T}_h^\beta$ and let $\Pi^\beta \bv_h^\beta = (v_{1h}^\beta,v_{2h}^\beta, \Pi_3^\beta v_{3h}^\beta)$.
The interpolation estimate and Poincar\'e inequality show
\begin{equation}
\label{ineq-norm}
\| \Pi^\beta {\bv}_h^\beta \|_{1,\Gamma} \leq \| \Pi^\beta {\bv}_h^\beta \|_{1,\beta} \leq C  \left(\sum_{I=1}^2|v_{Ih}^\beta|_{1,\beta}^2 + |v_{3h}^\beta|_{2,h}^2\right)^{\frac{1}{2}},
\end{equation}
for all $\bv_h^\beta \in W^\beta+W_{h}^\beta$.

To prove Assumption (A2), Lemma \ref{lem:interface-term} below will be used which is based on the estimate in Lemma \ref{lem:interface-term-alpha}. Recall $(\bullet,\bullet)_{\partial \mathcal{T}_h^\alpha}$ in \eqref{def:interface-product} and $\Sigma_h^\alpha$ in \eqref{def-sigmah}.

\begin{lemma}
\label{lem:interface-term-alpha} 
Given $\bsig_h^\alpha \in \tilde{\Sigma}^\alpha + \Sigma_h^\alpha$ and $\bv^\alpha \in H^1(\alpha;\mathbb{R}^3)$, it holds that
$$
(\bsig_h^\alpha \bn^\alpha, \bv^\alpha)_{\partial \mathcal{T}_h^\alpha}
\leq C \| \bsig_h^\alpha \|_{\divt,h} \| \bv^\alpha \|_{1,\alpha}.
$$  
\end{lemma}

\begin{proof}
The proof can be derived by integrating by parts on each element $K^\alpha \in \mathcal{T}_h^\alpha$ and using the Cauchy-Schwarz inequality as follows:
$$
\begin{aligned}
&(\bsig_h^\alpha \bn^\alpha, {\bv}^\alpha)_{\partial \mathcal{T}_h^\alpha} = \sum_{K^\alpha \in \mathcal{T}_h^\alpha} (\divt \bsig_h^\alpha, {\bv}^\alpha)_{K^\alpha} +     \sum_{K^\alpha \in \mathcal{T}_h^\alpha} ( \bsig_h^\alpha, \bvar^\alpha(\bv^\alpha))_{K^\alpha} \\
& \quad \leq \sum_{K^\alpha \in \mathcal{T}_h^\alpha} \| \divt \bsig_h^\alpha\|_{0,K^\alpha}\|{\bv}^\alpha\|_{0,K^\alpha} +     \sum_{K^\alpha \in \mathcal{T}_h^\alpha} \| \bsig_h^\alpha\|_{0,K^\alpha} \| \bvar^\alpha(\bv^\alpha) \|_{0,K^\alpha} \\
&\quad  \leq \| \divt_h \bsig_h^\alpha \|_{0,\alpha} \| \bv^\alpha\|_{0,\alpha} + C \| \bsig_h^\alpha\|_{0,\alpha} \| \bv^\alpha\|_{1,\alpha} \leq C \| \bsig_h^\alpha\|_{\divt,h} \|\bv^\alpha\|_{1,\alpha}.
\end{aligned}
$$
\end{proof}

\begin{lemma} 
\label{lem:interface-term}
Given $\bsig_h^\alpha \in \tilde{\Sigma}^\alpha + \Sigma_h^\alpha$ and $\bv^\beta \in H^1(\beta;\mathbb{R}^3)$, it holds that
$$
(\bsig_h^\alpha \bn^\alpha, \bv^\beta)_{\Gamma} \leq C \| \bsig_h^\alpha \|_{\divt,h} \| \bv^\beta \|_{1,\Gamma}.
$$    
\end{lemma}

\begin{proof}
Let $\hat{\bv}^\beta \in H^1(\alpha;\mathbb{R}^3)$ denote the bounded extension of the restriction of $\bv^\beta$ on $\Gamma$, namely, 
\begin{equation}
\label{estimate-extension}
\| \hat{\bv}^\beta \|_{1,\alpha} \leq  C \| \bv^\beta\|_{1,\Gamma}.  
\end{equation}
Recall the Scott-Zhang operator $\Pi_{SZ}^\alpha$  with \eqref{SZ}. A triangle inequality leads to
\begin{equation}
\label{estimate-SZ}
\| \Pi_{SZ}^\alpha \hat{\bv}^\beta \|_{1,\alpha} \leq \| \hat{\bv}^\beta - \Pi_{SZ}^\alpha \hat{\bv}^\beta \|_{1,\alpha} + \| \hat{\bv}^\beta\|_{1,\alpha}
\leq  C \| \hat{\bv}^\beta\|_{1,\alpha}.   
\end{equation}
Note that
\begin{equation}
\label{eq-decomposition}
\begin{aligned}
( \bsig_h^\alpha \bn^\beta, \bv^\beta )_{\Gamma} & = ( \bsig_h^\alpha \bn^\alpha, \hat{\bv}^\beta )_{\Gamma} \\
& = (\bsig_h^\alpha \bn^\alpha, \Pi_{SZ}^\alpha \hat{\bv}^\beta)_{\Gamma}
 + (\bsig_h^\alpha \bn^\alpha, \hat{\bv}^\beta -\Pi_{SZ}^\alpha  \hat{\bv}^\beta  )_{\Gamma}.    
\end{aligned}
\end{equation}
The definitions of $\tilde{\Sigma}^\alpha$ in \eqref{def:tSigma} and $\Sigma_h^\alpha$ in \eqref{def-sigmah} imply
$$
(\bsig_h^\alpha \bn^\alpha, \Pi_{SZ}^\alpha \hat{\bv}^\beta)_{\Gamma} = (\bsig_h^\alpha \bn^\alpha, \Pi_{SZ}^\alpha \hat{\bv}^\beta)_{\partial \mathcal{T}_h^\alpha}.
$$
This, Lemma \ref{lem:interface-term-alpha} and \eqref{estimate-extension}--\eqref{estimate-SZ} give
\begin{equation}
\label{case1-1}
\begin{aligned}
(\bsig_h^\alpha \bn^\alpha, \Pi_{SZ}^\alpha \hat{\bv}^\beta)_{\Gamma} &= 
 (\bsig_h^\alpha \bn^\alpha, \Pi_{SZ}^\alpha \hat{\bv}^\beta)_{\partial \mathcal{T}_h^\alpha} \leq C \| \bsig_h^\alpha \|_{\divt,h} \| \Pi_{SZ}^\alpha \hat{\bv}^\beta\|_{1,\alpha} \\
 & \leq C \| \bsig_h^\alpha \|_{\divt,h} \| \hat{\bv}^\beta\|_{1,\alpha} \leq C \| \bsig_h^\alpha \|_{\divt,h} \| \bv^\beta\|_{1,\Gamma}.    
\end{aligned}
\end{equation}
Recall the definition of $\mathcal{F}_{\Gamma}^\alpha$ in \eqref{def:F-gamma}. For any $F \in \mathcal{F}_{\Gamma}^\alpha$ with $F=K^\alpha \cap \Gamma$, the Cauchy-Schwarz inequality, the trace inequality and \eqref{SZ} give
$$
\begin{aligned}
|(\bsig_h^\alpha \bn^\alpha,  \hat{\bv}^\beta - \Pi_{SZ}^\alpha \hat{\bv}^\beta )_{F} |
& \leq \|\bsig_h^\alpha \|_{0,F} \|  \hat{\bv}^\beta - \Pi_{SZ}^\alpha \hat{\bv}^\beta \|_{0,F} \\
&\leq  C\|\bsig_h^\alpha \|_{0,F}  \|  \hat{\bv}^\beta - \Pi_{SZ}^\alpha \hat{\bv}^\beta \|_{0,K^\alpha}^{\frac{1}{2}}\|  \hat{\bv}^\beta - \Pi_{SZ}^\alpha \hat{\bv}^\beta \|_{1,K^\alpha}^{\frac{1}{2}} \\ 
&\leq C h_{F}^{\frac{1}{2}}\| \bsig_h^\alpha \|_{0,F} \| \hat{\bv}^\beta\|_{1,S_{K^\alpha}}.
\end{aligned}
$$
This, the Cauchy-Schwarz inequality and \eqref{estimate-extension} lead to
\begin{equation}
\label{case1-2}
\begin{aligned}
|( \bsig_h^\alpha \bn^\alpha,  \hat{\bv}^\beta - \Pi_{SZ}^\alpha \hat{\bv}^\beta )_{\Gamma}| 
& \leq  | \sum_{F\in \mathcal{F}_\Gamma^\alpha} ( \bsig_h^\alpha \bn^\alpha,  \hat{\bv}^\beta - \Pi_{SZ}^\alpha \hat{\bv}^\beta )_{F} |  \\
& \leq C \| \bsig_h^\alpha \|_{\divt,h}  \| \hat{\bv}^\beta\|_{1,\alpha} \leq C \| \bsig_h^\alpha \|_{\divt,h}  \| \bv^\beta \|_{1,\Gamma}.  
\end{aligned}    
\end{equation}
The combination of \eqref{eq-decomposition}--\eqref{case1-2} completes the proof.
\end{proof}

Lemma \ref{lem:interface-term}   holds without any mesh assumptions. This lemma is crucial to the error analysis for the nonconforming finite element method on non-matching meshes and the subsequent theorem.

\begin{theorem}
\label{thm:wellposed-nonconforming-discrete}
Problem \eqref{pro: weak-discrete-nc} is well-posed. 
\end{theorem}

\begin{proof}
According to Theorem \ref{thm:conti-nonconforming-error}, the well-posedness of \eqref{pro: weak-discrete-nc} follows from
Assumptions $\mathrm{(A1)}$--$\mathrm{(A3)}$.
It is easy to verify Assumption $\mathrm{(A1)}$. 
Note that $\Pi^\beta \bv_h^\beta \in H^1(\beta;\mathbb{R}^3)$ for all $\bv_h^\beta \in W^\beta + W_h^\beta$, the boundedness requirements in Assumption $\mathrm{(A2)}$ can be verified by Lemma \ref{lem:interface-term}. The inverse estimate shows
\begin{equation}
\label{ineq:core-discrete}
\| \btau_h^\alpha\|_{0,\alpha}^2 +  \sum_{F\in \mathcal{F}_\Gamma^\alpha} h_F \| \btau_h^\alpha \bn^\alpha\|_{0,F}^2 \leq C \| \btau_h^\alpha\|_{0,\alpha}^2 \text{ for all } \btau_h^\alpha \in \Sigma_h^\alpha.     
\end{equation}
Since $\divt_h \Sigma_h^\alpha \subset V_h^\alpha$ \cite{arnold2014nonconforming}, $\btau_h^\alpha \in B_h^{nc}$ implies $\divt_h \btau_h^\alpha =0$. 
This, the definition of $\|\bullet\|_{\divt,h}$ in \eqref{def:divh} and \eqref{ineq:core-discrete} lead to 
\begin{equation}
\label{eq:4.801}
(\mathcal{C}_0^{-1} \btau_h^\alpha, \btau_h^\alpha)_\alpha \geq C \| \btau_h^\alpha \|_{\divt,h} \text{ for all } \btau_h^\alpha \in B_h^{nc}.    
\end{equation}
For any $\bv_h^\alpha \in V_h^\alpha$, there exists a $\boldsymbol{q}^\alpha \in H^1(\alpha;\mathbb{S}) \cap \Sigma^\alpha$ such that $\divt_h \boldsymbol{q}^\alpha = \bv_h^\alpha$ and
$\| \boldsymbol{q}^\alpha\|_{1,\alpha}  \leq C \| \bv_h^\alpha\|_{0,\alpha}$, see e.g., \cite{arnold2008finite, pauly2022hilbert}. Let $P_h^\alpha$ be the $L^2$ projection operator onto $V_h^\alpha$.
Following similar procedures as in \cite{gopalakrishnan2011symmetric, arnold2014nonconforming} and noting that \eqref{ineq:core-discrete} for $\Sigma_h^\alpha$, there exists a Fortin operator $I_h^\alpha:H^1(\alpha;\mathbb{S})\cap \Sigma^\alpha \rightarrow \Sigma_h^\alpha$ satisfying $\divt_h I_h^\alpha \bsig^\alpha = P_h^\alpha \divt_h \bsig^\alpha$ for all $\bsig^\alpha \in H^1(\alpha; \mathbb{S})$ and $\| I_h^\alpha \bsig^\alpha \|_{\divt,h} \leq C \| \bsig^\alpha\|_{1,\alpha}$. Thus, for all $\bv_h^\alpha \in V_h^\alpha$, by taking $\btau_h^\alpha = I_h^\alpha \boldsymbol{q}^\alpha$, it holds that
\begin{equation}
\label{eq:4.802}
\inf_{\bv_h^\alpha \in V_h^\alpha} \sup_{\btau_h^\alpha \in \Sigma_h^\alpha} \frac{(\divt \btau_h^\alpha, \bv_h^\alpha)_\alpha}{\| \bv_h^\alpha\|_{0,\alpha} \| \btau_h^\alpha\|_{\divt,h}} \geq\frac{1}{C} \inf_{\bv_h^\alpha \in V_h^\alpha}  \frac{\|\bv_h^\alpha\|_{0,\alpha}^2 }{\|\bv_h^\alpha\|_{0,\alpha}^2 } \geq \frac{1}{C} .       
\end{equation} 
The combination of \eqref{eq:4.801} and \eqref{eq:4.802} proves Assumption $\mathrm{(A3)}$, which completes the proof.
\end{proof}

\begin{remark}
As noted in Remark \ref{remark:A2}, introducing the subspace $\tilde{\Sigma}^\alpha$ in Assumption $\mathrm{(A2)}$ is not required to analyze the well-posedness of \eqref{pro: weak-discrete-nc}. Accordingly, define the following norm: 
$$
\| \btau_h^\alpha\|_{\divt,h}^2 \coloneqq \| \btau_h^\alpha\|_{0,\alpha}^2 + \| \divt_h \btau_h^\alpha \|_{0,\alpha}^2 \text{ for all } \btau_h^\alpha \in \Sigma_h^\alpha.     
$$
Note that this definition is equivalent to the definition of the norm $\| \bullet \|_{\divt,h}^2$ given in \eqref{def:divh} for all $\btau_h^\alpha \in \Sigma_h^\alpha$. Following similar procedures, one can obtain the well-posedness.
\end{remark}

For the simplicity of error analysis, suppose the following regularity assumptions \cite{huang2006finite} 
\begin{equation}
\label{regularity}    
\bu^\alpha \in H^2(\alpha;\mathbb{R}^3),\, \bu^\beta \in H^2(\beta;\mathbb{R}^2) \times H^3(\beta), \, \bft^\alpha \in H^1(\alpha;\mathbb{R}^3),\, \bft^\beta \in L^2(\beta,\mathbb{R}^3).
\end{equation} 
This leads to $\bsig^\alpha \in H^s(\alpha;\mathbb{S})$ with $s=1$, i.e., $\bsig^\alpha \in \tilde{\Sigma}^\alpha$.
The following lemma will be used to obtain the error estimate of   \eqref{pro: weak-discrete-nc}. Recall $W_2^\beta = H_0^2(\beta)$ and $W_{2h}^\beta$ is the Morley finite element space. In the following lemma and subsequent context, with a slight abuse of notation, $\bM^\beta(v^\beta)$ and $\bK^\beta(v^\beta)$ denote $\bM^\beta(\bv^\beta)$ and $\bK^\beta(\bv^\beta)$ with $\bv^\beta = (0,0,v^\beta)$. The same convention applies to $\bM_h^\beta(v^\beta)$ and $\bK_h^\beta(v^\beta)$.

\begin{lemma}[\cite{hu2014new, gudi2010new}]
\label{lem-E32}
Given $g^\beta \in L^2(\beta)$, let $w^\beta \in W_2^\beta$ be the solution of 
$$
(\bM^\beta(w^\beta), \bK^\beta(v^\beta))_\beta = (g^\beta, v^\beta)_\beta \text{ for all } v^\beta \in W_2^\beta.
$$
Then it holds that
$$
\begin{aligned}
& |(\bM_h^\beta (w^\beta), \bK_h^\beta(w_h^\beta))_\beta -(g^\beta,w_h^\beta)_{\beta}| \\
& \leq\quad C\left\{\inf_{\upsilon_h\in W_{2h}^\beta} |w^\beta-\upsilon_h|_{2,h}+\left(\sum_{K^\beta \in \mathcal{T}_h^\beta}h_{K^\beta}^4\inf_{q\in {P}_{k}(K^\beta)}\|g^\beta-q\|_{0,K^\beta}^2\right)^{\frac{1}{2}}\right\} |w_h^\beta|_{2,h}
\end{aligned}
$$
for all $w_h^\beta \in W_{2h}^\beta$ with any $k \geq 0$.
\end{lemma}

\begin{theorem} 
\label{thm:convergence-rates-nonconforming}
Under the regularity assumptions \eqref{regularity},  $(\bsig_h^\alpha, \bu_h^\alpha, \bu_h^\beta)$ as the exact solution of \eqref{pro: weak-discrete-nc} admits the following convergence rate
$$
\begin{aligned}
\interleave ( \bsig^\alpha - \bsig_h^\alpha, \bu^\beta- \bu_h^\beta) \interleave_{h} + \| \bu^\alpha - \bu_h^\alpha \|_{0,\alpha} \leq C(h_\alpha + h_\beta).    
\end{aligned}
$$
\end{theorem}

\begin{proof}
According to Theorem \ref{thm:conti-nonconforming-error} and \ref{thm:wellposed-nonconforming-discrete}, it only needs to prove Assumptions $\mathrm{(A4)}$--$\mathrm{(A5)}$. The proof is divided into four steps.

\textbf{Step 1} proves Assumption $\mathrm{(A4)}$. Recall the Fortin operator $I_h^\alpha$ and the $L^2$ projection operator $P_h^\alpha$ in the proof of Theorem \ref{thm:wellposed-nonconforming-discrete}. For all $K^\alpha \in \mathcal{T}_h^\alpha$, it holds that
\begin{equation}
\label{A4:1}
\| \btau^\alpha - I_h^\alpha \btau^\alpha \|_{0,K^\alpha} + h_\alpha | \btau^\alpha - I_h^\alpha \btau^\alpha|_{1,K^\alpha}\leq C h_\alpha \| \btau^\alpha\|_{1,K^\alpha}.
\end{equation}
The trace inequality and \eqref{A4:1} give
\begin{equation}
\label{A4:3}
\begin{aligned}
h_F \| \btau^\alpha - I_h^\alpha \btau^\alpha\|_{0,F}^2 & \leq Ch_F \| \btau^\alpha - I_h^\alpha \btau^\alpha\|_{0,K^\alpha} \| \btau^\alpha - I_h^\alpha \btau^\alpha\|_{1,K^\alpha}\\
& \leq C h_{K^\alpha}^2 \| \btau^\alpha\|_{1,K^\alpha}^2 \text{ for all } F \in \mathcal{F}_\Gamma^\alpha \text{ with } F = K^\alpha \cap \Gamma.
\end{aligned}
\end{equation}
Since $\divt_h  I_h^\alpha = P_h^\alpha \divt$, it holds that
\begin{equation}
\label{A4:2}
\|\divt \bsig^\alpha - \divt_h  I_h^\alpha \bsig^\alpha \|_{0,\alpha} \leq C h_\alpha | \divt \bsig^\alpha|_{1,\alpha} = C h_\alpha |\bft^\alpha|_{1,\alpha}.     
\end{equation}
Recall the definition of $\|\bullet\|_{\divt,h}$ in \eqref{def:divh}. The combination of \eqref{A4:1}--\eqref{A4:2} leads to
\begin{equation}
\label{interpolation-alpha}
\| \btau^\alpha - I_h^\alpha \btau^\alpha \|_{\divt,h} \leq C h_\alpha (  | \bft^\alpha |_{1,\alpha} +  | \btau^\alpha |_{1,\alpha} ).     
\end{equation}
Let $I_{Ih}^\beta$ and $I_{3h}^\beta$ be the standard interpolation operators related to $W_{1h}^\beta$ and $W_{2h}^\beta$ in \cite{brenner2008mathematical} and  \cite{li2014new}, respectively.  The following error estimates hold
\begin{equation}
\label{interpolation-beta}
\sum_{I=1}^2| v_{I}^\beta - I_{Ih}^\beta v_{I}^\beta |_{1,\beta} + | v_{3}^\beta - I_{3h}^\beta v_{3}^\beta |_{2,h} \leq C h_\beta \left( \sum_{I=1}^2 | v_{I}^\beta|_{2,\beta} + |v_3^\beta|_{3,\beta} \right).
\end{equation}
This and \eqref{interpolation-alpha} give 
\begin{equation}
\label{eq:41201}
\begin{aligned}
& \inf_{(\btau_h^\alpha, \bv_h^\beta)\in \Sigma_h^\alpha \times W_h^\beta} \interleave ( \bsig^\alpha - \btau_h^\alpha, \bu^\beta- \bv_h^\beta) \interleave_{h}  \\
& \leq \left( \| \bsig^\alpha - I_h^\alpha\bsig^\alpha \|_{\divt,h}^2 + \sum_{I=1}^2| v_{I}^\beta - I_{Ih}^\beta v_{I}^\beta |_{1,\beta}^2 + | v_{3}^\beta - I_{3h}^\beta v_{3}^\beta |_{2,h}^2\right)^\frac{1}{2}  \\
& \leq C h_\alpha (|\bft^\alpha|_{1,\alpha}+  | \bsig^\alpha |_{1,\alpha} ) + C h_\beta\left( \sum_{I=1}^2 | v_{I}^\beta|_{2,\beta} + |v_3^\beta|_{3,\beta} \right).     
\end{aligned}   
\end{equation}
It holds that
\begin{equation}
\label{eq:41202}
\inf_{\bv_h^\alpha \in V_h^\alpha} \| \bu^\alpha - \bv_h^\alpha \|_{0,\alpha}  \leq  \| \bu^\alpha - P_h^\alpha \bu^\alpha\|_{0,\alpha } \leq C h_\alpha |\bu^\alpha|_{2,\alpha}.
\end{equation}
The combination of \eqref{eq:41201} and \eqref{eq:41202} proves Assumption $\mathrm{(A4)}$ as follows:
$$
\begin{aligned}
& \inf_{(\btau_h^\alpha, \bv_h^\beta)\in \Sigma_h^\alpha \times W_h^\beta} \interleave ( \bsig^\alpha - \btau_h^\alpha, \bu^\beta- \bv_h^\beta) \interleave_{h} +\inf_{\bv_h^\alpha \in V_h^\alpha} \| \bu^\alpha - \bv_h^\alpha \|_{0,\alpha} \\
& \leq C h_\alpha (|\bft^\alpha|_{1,\alpha}+|\bu^\alpha|_{2,\alpha}) + C h_\beta\left( \sum_{I=1}^2 | v_{I}^\beta|_{2,\beta} + |v_3^\beta|_{3,\beta} \right).
\end{aligned}
$$

\textbf{Step 2} analyzes $T_0$ in Assumption $\mathrm{(A5)}$. Suppose that $\btau_h^\alpha\in\Sigma_h^\alpha$ and $\bv^\beta_h=(v_{1h}^\beta,v_{2h}^\beta,v_{3h}^\beta)\in W_h^\beta$. Note that $W_{1h}^\beta \subset W_1^\beta$ is a conforming continuous piecewise linear finite element space. Recall the definition of $\bV$ in \eqref{def: bV}. The choice of $(\bv^\alpha,\bv^\beta) \in \bV$ with $\bv^\beta =  (v_{1h}^\beta,v_{2h}^\beta, 0)$ and $\bv^\alpha = \bv^\beta$ on $\Gamma$ in \eqref{eq:eqivalence2} combined with $\bsig^\alpha \bn^\alpha=0$ on $\partial \alpha \backslash \Gamma$ yields
\begin{equation*}
\label{eq: 41202}
(\sigma_{IJ}^\beta(\bu^\beta), \varepsilon_{IJ}^\beta(\bv_h^\beta))_\beta + ( \sigma_{Ij}^\alpha n_{j}^\alpha, {v}_{Ih}^\beta )_\Gamma - (f_{I}^\beta, v_{Ih}^\beta)_\beta = 0.      
\end{equation*}  
A substitution of this into $T_0$ from \eqref{def:T0} shows
$$
T_0 = a_h(\bsig^\alpha, \bu^\beta; \btau_h^\alpha, \bv_h^\beta) + b_h(\bu^\alpha; \btau_h^\alpha, \bv_h^\beta) - (\bft^\beta, \bv_h^\beta) \coloneqq T_{01} + T_{02} +T_{03}
$$
where 
$$
\begin{aligned}
T_{01} &\coloneqq (\mathcal{C}_0^{-1} \bsig^\alpha, \btau_h^\alpha)_\alpha - ( \btau_h^\alpha \bn^\alpha, \bu^\beta)_{\Gamma} + (\divt_h \btau_h^\alpha, \bu^\alpha)_\alpha,\\
T_{02} & \coloneqq 
(\bM_h^\beta ( \bu^\beta), \bK_h^\beta (\bv_{h}^\beta))_\beta + (\sigma_{3j}^\alpha n_{j}^\alpha, v_{3h}^\beta)_{\Gamma} - (f_3^\beta, v_{3h}^\beta)_\beta, \\
T_{03} & \coloneqq (\btau_h^\alpha \bn^\alpha, \bu^\beta - \Pi^\beta \bu^\beta)_{\Gamma} + (\sigma_{3j}^\alpha n_{j}^\alpha, \Pi_3^\beta v_{3h}^\beta - v_{3h}^\beta)_{\Gamma}.
\end{aligned}
$$

\textbf{Step 3} estimates $T_{01}$ and $T_{02}$. Recall the Scott-Zhang interpolation operator $\Pi_{SZ}^\alpha$ in \eqref{SZ}. The integration by parts, \eqref{def-sigmah} and $\bu^\alpha = \bu^\beta$ on $\Gamma$ show 
    \begin{equation}
    \label{T010}
    \begin{aligned}
    T_{01} = (\btau_h^\alpha \bn^\alpha, \bu^\alpha)_{\partial \mathcal{T}_h^\alpha} - (\btau_h^\alpha \bn^\alpha, \bu^\beta)_\Gamma  = (\btau_h^\alpha \bn^\alpha, \bu^\alpha - \Pi_{SZ}^\alpha \bu^\alpha )_{\partial \mathcal{T}_h^\alpha\backslash \Gamma}.
    \end{aligned} 
    \end{equation}
Given $K^\alpha \in \mathcal{T}_h^\alpha$, for any $F \in \partial K^\alpha$, the Cauchy-Schwarz inequality, the trace inequality, \eqref{SZ} and the inverse estimate give
$$
\begin{aligned}
| ( \btau_h^\alpha \bn^\alpha, \bu^\alpha - \Pi_{SZ}^\alpha \bu^\alpha)_{F} |  &\leq \| \btau_h^\alpha \|_{0,F} \| \bu^\alpha - \Pi_{SZ}^\alpha \bu^\alpha \|_{0,F} \\
&\leq C \| \btau_h^\alpha \|_{0,F} \| \bu^\alpha - \Pi_{SZ}^\alpha \bu^\alpha \|_{0,K^\alpha}^{\frac{1}{2}} \| \bu^\alpha - \Pi_{SZ}^\alpha \bu^\alpha \|_{1,K^\alpha}^{\frac{1}{2}} \\
 & \leq C h_{K^\alpha}^{\frac{3}{2}} \| \btau_h^\alpha \|_{0,F} \| \bu^\alpha \|_{2,S_{K^\alpha}} \leq C h_{\alpha} \| \btau_h^\alpha \|_{0,K^\alpha} \| \bu^\alpha \|_{2,S_{K^\alpha}}. 
\end{aligned}
$$
Recall the definition of $\mathcal{F}_\Gamma^\alpha$ in \eqref{def:F-gamma}. This and the Cauchy-Schwarz inequality lead to 
\begin{equation}
\label{T012}
\begin{aligned}
|( \btau_h^\alpha \bn^\alpha, \bu^\alpha - \Pi_{SZ}^\alpha \bu^\alpha)_{\partial \mathcal{T}_h^\alpha \backslash \Gamma}| &= |\sum_{K^\alpha \in \mathcal{T}_h^\alpha}\sum_{F\in \partial K^\alpha \backslash \mathcal{F}_\Gamma^\alpha} ( \btau_h^\alpha \bn^\alpha, \bu^\alpha - \Pi_{SZ}^\alpha \bu^\alpha)_{F} | \\
& \leq C h_\alpha \|\btau_h^\alpha\|_{0,\alpha} \|\bu^\alpha\|_{2,\alpha}.    
\end{aligned}
\end{equation}
The combination of \eqref{T010}--\eqref{T012} leads to 
\begin{equation}
\label{T01}
|T_{01}| \leq C h_{\alpha} \interleave (\btau_h^\alpha, \bv_h^\beta)\interleave_{h} \| \bu^\alpha \|_{2,\alpha}.
\end{equation}
Under the regularity assumptions \eqref{regularity}, it holds that $f_3^\beta - \sigma_{3j}^\alpha n_j^\alpha \in L^2(\beta)$ with $\sigma_{3j}^\alpha n_j^\alpha$ denotes its zero extension on $\beta$. Problem \eqref{pro:displacement-based} gives that $u_3^\beta$ is the solution of 
$$
(\bM^\beta(u_3^\beta), \bK^\beta(v^\beta))_\beta = (f_3^\beta - \sigma_{3j}^\alpha n_j^\alpha, v^\beta)_\beta \text{ for all } v^\beta \in W_2^\beta.
$$
The choice of $g^\beta = f_3^\beta - \sigma_{3j}^\alpha n_j^\alpha$, $w^\beta = u_3^\beta$ and $w_h = v_{3h}^\beta$ in Lemma \ref{lem-E32} and the standard interpolation error estimations give
$$
\begin{aligned}
|T_{02}| & \leq 
C h_\beta \left\{ |u_3^\beta|_{3,\beta}+\left(\sum_{K^\beta\in\mathcal{T}_h^\beta}h_{K^\beta}^2\inf_{q\in {P}_{k}(K^\beta)}\|f_3^\beta - \sigma_{3j}^\alpha n_j^\alpha-q\|_{0,K^\beta}^2\right)^{\frac{1}{2}}\right\} |v_{3h}^\beta|_{2,h}\\
& \leq C h_\beta \left( |u_3^\beta|_{3,\beta}+ h_\beta \| f_3^\beta\|_{0,\beta} + h_\beta  \| \bsig^\alpha \|_{1,\alpha}  
\right) |v_{3h}^\beta|_{2,h}.
\end{aligned}
$$
This gives
\begin{equation}
\label{T02}
|T_{02}| \leq C h_\beta \interleave (\btau_h^\alpha, \bv_h^\beta)\interleave_{h} \left( |u_3^\beta|_{3,\beta}+ \| f_3^\beta\|_{0,\beta} +  \| \bu^\alpha \|_{2,\alpha}  \right).
\end{equation}

\textbf{Step 4} estimates $T_{03}$ and concludes the proof of Assumption $\mathrm{(A5)}$.
Lemma \ref{lem:interface-term} and $\bu^\beta - \Pi^\beta \bu^\beta \in H^1(\beta;\mathbb{R}^3)$ give
\begin{equation}
\label{T031}
\begin{aligned}
(\btau_h^\alpha\bn^\alpha, \bu^\beta - \Pi^\beta \bu^\beta)_{\Gamma} &\leq C \| \btau_h^\alpha\|_{\divt,h} \|\bu^\beta - \Pi^\beta \bu^\beta\|_{1,\Gamma}\\
&\leq C h_\beta \interleave (\btau_h^\alpha, \bv_h^\beta)\interleave_{h} \| \bu^\beta \|_{2,\beta}.
\end{aligned}
\end{equation}
The Cauchy-Schwarz inequality, the interpolation estimate and the trace inequality lead to 
\begin{equation}
\label{T032}
\begin{aligned}
(\sigma_{3j}^\alpha n_{j}^\alpha, \Pi_3^\beta v_{3h}^\beta - v_{3h}^\beta)_{\Gamma} & \leq \| \sigma_{3j}^\alpha n_{j}^\alpha\|_{0,\Gamma}    \| \Pi_3^\beta v_{3h}^\beta - v_{3h}^\beta\|_{0,\Gamma} \\
& \leq C h_\beta^2 \| \sigma_{3j}^\alpha n_{j}^\alpha\|_{0,\Gamma}   \| v_{3h}^\beta \|_{2,h} \\
& \leq C h_\beta^2 \interleave (\btau_h^\alpha, \bv_h^\beta)\interleave_{h} \| \bsig^\alpha\|_{1,\alpha}.
\end{aligned}
\end{equation}
The combination of \eqref{T031}--\eqref{T032} gives to 
\begin{equation}
\label{T03}
|T_{03}| \leq C h_{\beta} \interleave (\btau_h^\alpha, \bv_h^\beta)\interleave_{h} (\| \bu^\beta \|_{2,\beta}+  \| \bu^\alpha\|_{2,\alpha} ).    
\end{equation}
The summation of the estimates in \eqref{T01}, \eqref{T02} and \eqref{T03}  gives
$$
\frac{|T_0|}{ \interleave (\btau_h^\alpha, \bv_h^\beta)\interleave_{h}} \leq Ch_\alpha \|\bu^\alpha\|_{2,\alpha} + Ch_\beta \left( \|f_3^\beta\|_{0,\beta} + \sum_{I=1}^2 \|u_I^\beta\|_{2,\beta}+ \|u_3^\beta\|_{3,\beta}+\|\bu^\alpha\|_{2,\alpha} \right).
$$
The proof is completed.
\end{proof}


\section{Numerical experiments}

This section presents numerical experiments to validate the theoretical results. Additionally, the problem is reduced to an interface problem using the domain decomposition, which is solved effectively by the conjugate gradient iteration.

\subsection{Convergence rate test}

Let $\alpha = (-\frac{1}{2},\frac{1}{2})^2\times (0,1)$ and $\beta = (-1,1)^2 \times \{0\}$, which are rigidly connected along the interface $\Gamma = (-\frac{1}{2},\frac{1}{2})^2 \times \{0 \}$. Consider the following  
exact solution for the equation \eqref{equil}
$$
\bu^\alpha = \left(  
\begin{aligned}
& \sin(\pi x) \sin(\pi y)(1+z^2) \\
& \sin(\pi x) \sin(\pi y)(1+z^2) \\
& (1-x^2)^2(1-y^2)^2(1+z^2)
\end{aligned}
\right),
\quad
\bu^\beta = \left(  
\begin{aligned}
& \sin(\pi x) \sin(\pi y) \\
& \sin(\pi x) \sin(\pi y) \\
& (1-x^2)^2(1-y^2)^2
\end{aligned}
\right).
$$
The forces $\bft^{\alpha},  \bft^{\beta}$ can be obtained from above exact solution correspondingly. Let $E_{\alpha}=100, {\nu}_\alpha = {\nu}_{\beta} = 0.3, t_{\beta} = 0.02$, and choose $E_\beta$ such that the flexible rigidity $D_\beta = \frac{E_\beta t_\beta^3}{12(1-\nu_\beta^2)}=1$. 

\begin{figure}[!h]
  \centering
  \begin{subfigure}[b]{0.49\textwidth}
    \includegraphics[width=\textwidth]{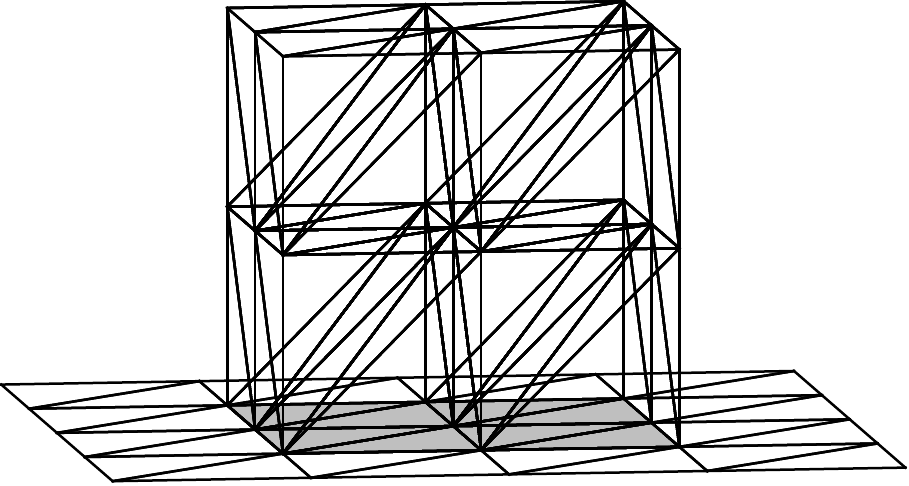} 
    \caption{Initial mesh with matching mesh.}
    \label{fig:initialmesh1}
  \end{subfigure}
  \hfill 
  \begin{subfigure}[b]{0.49\textwidth}
    \includegraphics[width=\textwidth]{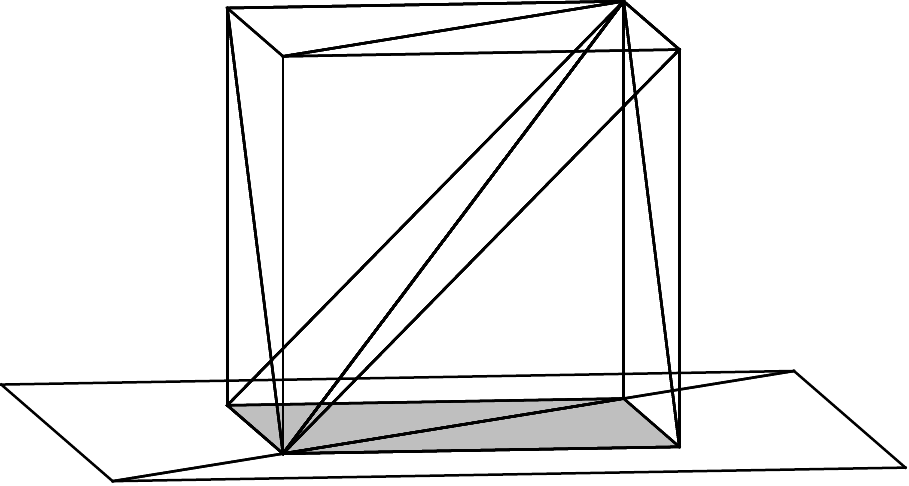} 
    \caption{Initial mesh with non-matching mesh.}
    \label{fig:initialmesh2}
  \end{subfigure}
  \caption{Initial meshes.}
\end{figure}

To compare with the displacement-based formulation, this paper uses the nonconforming finite element method in \cite{huang2006finite} for solving the coupled body-plate model. 
In the following, $u_{*}^\beta \coloneqq (u_{1}^\beta, u_2^\beta)$ and $u_{*h}^\beta \coloneqq (u_{1h}^\beta, u_{2h}^\beta)$ are introduced to simplify notations below. 
The numerical results are listed in Table \ref{tab:PPnonconforming}, which show that 
$$
\| \bu^\alpha -\bu_h^\alpha \|_{1,\alpha} +  | u_{*}^\beta - u_{*h}^\beta|_{1,\beta} + | u_{3}^\beta - u_{3h}^\beta|_{2,\beta} \leq C(h_\alpha + h_\beta).
$$

Table \ref{tab:MPconforming} and Table \ref{tab:MPconforming_nonmatch} show the convergence rates of $(\bsig^\alpha, \bu^\alpha, \bu^\beta)$ for the conforming finite element method presented in Sec \ref{sub:conformingelement} with matching and non-matching meshes, respectively. The initial mesh with matching and non-matching mesh on the interface are shown in Figure \ref{fig:initialmesh1} and \ref{fig:initialmesh2}, respectively. The mesh is uniformly refined. The numerical results show 
\begin{equation*}
\| (\bsig^\alpha-\bsig_h^\alpha, \bu^\beta-\bu_h^\beta) \interleave + \interleave \bu^\alpha -\bu_h^\alpha \|_{0,\alpha} \leq C (h_\alpha^4+h_\beta^4),
\end{equation*}
which coincide with Theorem \ref{theoremConvergence}.

Similarly, Table \ref{tab:MPnonconforming} and Table \ref{tab:MPnonconforming_nonmatch} show the convergence rates of $(\bsig^\alpha, \bu^\alpha, \bu^\beta)$ for the nonconforming finite element method presented in Sec \ref{sub:nonconformingelement} with matching and non-matching meshes, respectively. The numerical results show
$$
 \interleave ( \bsig^\alpha - \bsig_h^\alpha, \bu^\beta- \bu_h^\beta) \interleave_{h} + \| \bu^\alpha - \bu_h^\alpha \|_{0,\alpha} 
\leq C ( h_\alpha + h_\beta),
$$
which coincide with Theorem \ref{thm:convergence-rates-nonconforming}.

\begin{table}[!h]
  \centering
  \begin{tabular}{ccccccc}
    \toprule
     \multicolumn{7}{c}{$P_1$ Lagrange element for body.} \\
    \midrule
    Mesh & $\|\bsig^\alpha - \bsig_h^\alpha\|_{0,\alpha}$ & rate & $\| \bu^\alpha - \bu_h^\alpha\|_{0,\alpha}$ & rate & & \\
    \midrule
    1 & 7.87763E+02 & - & 6.15168E-01 & - \\
    2 & 4.41518E+02 & 0.84 & 1.98381E-01 & 1.63 \\
    3 & 2.31723E+02 & 0.93 & 5.82705E-02 & 1.77 \\
    4 & 1.18309E+02 & 0.97 & 1.58125E-02 & 1.88 \\
    \midrule
     \multicolumn{7}{c}{Vectorial $P_1$ Lagrange element for the plane elasticity model.} \\
    \midrule
    Mesh & $|u_{*}^\beta - u_{*h}^\beta |_{1,\beta}$ & rate & $\|u_{*}^\beta - u^\beta_{*h}\|_{0,\beta}$ & rate & & \\
    \midrule
   1 & 3.65458E+00 & - & 5.68435E-01 & - \\
    2 & 1.96858E+00 & 0.89 & 1.61584E-01 & 1.81 \\
    3 & 1.00158E+00 & 0.98 & 4.16259E-02 & 1.96 \\
    4 & 5.02941E-01 & 0.99 & 1.04835E-02 & 1.99 \\
    \midrule
     \multicolumn{7}{c}{$P_2$ Morley element for the plate bending model.} \\
    \midrule
    Mesh & $|u_3^\beta - u_{3h}^\beta |_{2,h}$ & rate & $| u_3^\beta -  u^\beta_{3h}|_{1,h}$ & rate & $\|u_3^\beta - u_{3h}^\beta \|_{0,\beta}$ & rate \\
    \midrule
    1 & 7.30964E+00 & - & 1.23055E+00 & - & 6.84019E-01 & - \\
    2 & 3.78862E+00 & 0.95 & 3.69399E-01 & 1.74 & 2.04226E-01 & 1.74 \\
    3 & 1.90860E+00 & 0.99 & 9.96592E-02 & 1.89 & 5.53819E-02 & 1.88 \\
    4 & 9.55713E-01 & 1.00 & 2.54901E-02 & 1.97 & 1.41778E-02 & 1.97 \\
    \bottomrule
  \end{tabular}
\caption{Displacement-based nonconforming method with matching mesh.}
\label{tab:PPnonconforming}
\end{table}

\begin{table}[ht!]
  \centering
  \begin{tabular}{ccccccc}
    \toprule
     \multicolumn{7}{c}{$P_4$--$P_3$ $H(\divt;\mathbb{S})$ element in \cite{Hu2015,HuZhang2015} for body.} \\
    \midrule
    Mesh & $\|\bsig^\alpha - \bsig_h^\alpha\|_{0,\alpha}$ & rate & $\| \bu^\alpha - \bu_h^\alpha\|_{0,\alpha}$ & rate & & \\
    \midrule
    1 & 1.10854E+00 & - & 5.04072E-03 & - \\
    2 & 3.53174E-02 & 4.97 & 2.89330E-04 & 4.12 \\
    3 & 1.13190E-03 & 4.96 & 1.81806E-05 & 3.99 \\
    \midrule
     \multicolumn{7}{c}{Vectorial $P_4$ Lagrange element for the plane elasticity model.} \\
    \midrule
    Mesh & $|u_{*}^\beta - u_{*h}^\beta |_{1,\beta}$ & rate & $\|u_{*}^\beta - u^\beta_{*h}\|_{0,\beta}$ & rate & & \\
    \midrule
    1 & 4.28286E-02 & - & 1.58633E-03 & - \\
    2 & 2.84600E-03 & 3.91 & 5.33692E-05 & 4.89 \\
    3 & 1.81003E-04 & 3.97 & 1.71439E-06 & 4.96 \\
    \midrule
     \multicolumn{7}{c}{$P_5$ Argyris element for plate bending model.} \\
    \midrule
    Mesh & $|u_3^\beta - u_{3h}^\beta |_{2,\beta}$ & rate & $| u_3^\beta - u^\beta_{3h}|_{1,\beta}$ & rate & $\|u_3^\beta - u_{3h}^\beta \|_{0,\beta}$ & rate \\
    \midrule
    1 & 1.09546E-01 & - & 6.73757E-03 & - & 6.45495E-04 & - \\
    2 & 6.29546E-03 & 4.12 & 1.73323E-04 & 5.28 & 7.99916E-06 & 6.33 \\
    3 & 3.46697E-04 & 4.18 & 4.36422E-06 & 5.31 & 1.04375E-07 & 6.26 \\
    \bottomrule
  \end{tabular}
\caption{Conforming method with matching mesh.}
\label{tab:MPconforming}
\end{table}

\begin{table}[ht!]
  \centering
  \begin{tabular}{ccccccc}
    \toprule
     \multicolumn{7}{c}{$P_4$--$P_3$ $H(\divt;\mathbb{S})$ element in \cite{Hu2015,HuZhang2015} for body.} \\
    \midrule
    Mesh & $\|\bsig^\alpha - \bsig_h^\alpha\|_{0,\alpha}$ & rate & $\| \bu^\alpha - \bu_h^\alpha\|_{0,\alpha}$ & rate & & \\
    \midrule
    1 & 5.18392E+01 & - & 8.76128E-01 & - \\
    2 & 2.77751E+00 & 4.22 & 3.41237E-02 & 4.68 \\
    3 & 2.74505E-01 & 3.34 & 3.38345E-04 & 6.66 \\
    4 & 1.33368E-02 & 4.36 & 1.84796E-05 & 4.19 \\
    \midrule
     \multicolumn{7}{c}{Vectorial $P_4$ Lagrange element for the plane elasticity model.} \\
    \midrule
    Mesh & $|u_{*}^\beta - u_{*h}^\beta |_{1,\beta}$ & rate & $\|u_{*}^\beta - u^\beta_{*h}\|_{0,\beta}$ & rate & & \\
    \midrule
    1 & 2.63977E+00 & - & 7.07376E-01 & - \\
    2 & 2.25157E-01 & 3.55 & 3.26771E-02 & 4.44 \\
    3 & 4.28286E-02 & 2.39 & 1.58632E-03 & 4.36 \\
    4 & 2.84600E-03 & 3.91 & 5.33691E-05 & 4.89 \\
    \midrule
     \multicolumn{7}{c}{$P_5$ Argyris element for plate bending model.} \\
    \midrule
    Mesh & $|u_3^\beta - u_{3h}^\beta |_{2,\beta}$ & rate & $|u_3^\beta - u_{3h}^\beta |_{1,\beta}$ & rate & $\|u_3^\beta - u_{3h}^\beta \|_{0,\beta}$ & rate \\
    \midrule
    1 & 7.50136E+00 & - & 1.92032E+00 & - & 8.25911E-01 & - \\
    2 & 1.25485E+00 & 2.58 & 1.58705E-01 & 3.60 & 3.54216E-02 & 4.54 \\
    3 & 1.09360E-01 & 3.52 & 6.60994E-03 & 4.59 & 5.55617E-04 & 5.99 \\
    4 & 6.29485E-03 & 4.12 & 1.72372E-04 & 5.26 & 6.48075E-06 & 6.42 \\
    \bottomrule
  \end{tabular}
\caption{Conforming method with non-matching mesh.}
\label{tab:MPconforming_nonmatch}
\end{table}

\begin{table}[ht!]
  \centering
  \begin{tabular}{ccccccc}
    \toprule
    \multicolumn{7}{c}{Nonconforming $H(\divt;\mathbb{S})$ in \cite{HuMa2018} for plate.} \\
    \midrule
    Mesh & $\|\bsig^\alpha - \bsig_h^\alpha\|_{0,\alpha}$ & rate & $\| \bu^\alpha - \bu_h^\alpha\|_{0,\alpha}$ & rate & & \\
    \midrule
    1 & 8.91246E+01 & - & 4.83575E-01 & - \\
    2 & 3.45424E+01 & 1.37 & 1.46539E-01 & 1.72 \\
    3 & 1.58159E+01 & 1.13 & 3.88831E-02 & 1.91 \\
    \midrule
    \multicolumn{7}{c}{Vectorial $P_1$ Lagrange element for the plane elasticity model.} \\
    \midrule
    Mesh & $|u_{*}^\beta - u_{*h}^\beta |_{1,\beta}$ & rate & $\|u_{*}^\beta - u^\beta_{*h}\|_{0,\beta}$ & rate & & \\
    \midrule
    1 & 3.65458E+00 & - & 5.68418E-01 & - \\
    2 & 1.96858E+00 & 0.89 & 1.61575E-01 & 1.81 \\
    3 & 1.00158E+00 & 0.98 & 4.16234E-02 & 1.96 \\
    \midrule
    \multicolumn{7}{c}{$P_2$ Morley element for plate bending model.} \\
    \midrule
    Mesh & $|u_3^\beta - u_{3h}^\beta |_{2,h}$ & rate & $| u_3^\beta - u^\beta_{3h}|_{1,h}$ & rate & $\|u_3^\beta - u_{3h}^\beta \|_{0,\beta}$ & rate \\
    \midrule
    1 & 7.24326E+00 & - & 1.23117E+00 & - & 7.10979E-01 & - \\
    2 & 3.73113E+00 & 0.96 & 3.51072E-01 & 1.81 & 2.05529E-01 & 1.79 \\
    3 & 1.89744E+00 & 0.98 & 9.28534E-02 & 1.92 & 5.40035E-02 & 1.93 \\
    \bottomrule
  \end{tabular}
    \caption{Nonconforming method with matching mesh.}
    \label{tab:MPnonconforming}
\end{table}

\begin{table}[ht!]
  \centering
  \resizebox{\textwidth}{!}{
  \begin{tabular}{ccccccc}
    \toprule
    \multicolumn{7}{c}{Nonconforming $H(\divt;\mathbb{S})$ in \cite{HuMa2018} for plate.} \\
    \midrule
    Mesh & $\|\bsig^\alpha - \bsig_h^\alpha\|_{0,\alpha}$ & rate & $\| \bu^\alpha - \bu_h^\alpha\|_{0,\alpha}$ & rate & & \\
    \midrule
    1 & 2.28331E+02 & - & 9.50977E-01 & - \\
    2 & 1.27772E+02 & 0.84 & 7.93196E-01 & 0.26 \\
    3 & 5.52674E+01 & 1.21 & 4.51941E-01 & 0.81 \\
    4 & 2.21096E+01 & 1.32 & 1.41301E-01 & 1.68 \\
    5 & 9.25351E+00 & 1.26 & 3.79345E-02 & 1.90 \\
    \midrule
    \multicolumn{7}{c}{Vectorial $P_1$ Lagrange element for the plane elasticity model.} \\
    \midrule
    Mesh & $|u_{*}^\beta - u_{*h}^\beta |_{1,\beta}$ & rate & $\|u_{*}^\beta - u^\beta_{*h}\|_{0,\beta}$ & rate & & \\
    \midrule
    1 & 5.29438E+00 & - & 1.86596E+00 & - \\
    2 & 6.50074E+00 & -0.30 & 7.61935E-01 & 1.29 \\
    3 & 3.65458E+00 & 0.83 & 5.68422E-01 & 0.42 \\
    4 & 1.96858E+00 & 0.89 & 1.61576E-01 & 1.81 \\
    5 & 1.00158E+00 & 0.98 & 4.16237E-02 & 1.96 \\
    \midrule
    \multicolumn{7}{c}{$P_2$ Morley element for plate bending model.} \\
    \midrule
    Mesh & $|u_3^\beta - u_{3h}^\beta |_{2,h}$ & rate & $| u_3^\beta -  u^\beta_{3h}|_{1,h}$ & rate & $\|u_3^\beta - u_{3h}^\beta \|_{0,\beta}$ & rate \\
    \midrule
    1 & 7.50136E+00 & - & 1.92032E+00 & - & 8.25911E-01 & - \\
    2 & 8.54794E+00 & -0.19 & 2.02470E+00 & -0.08 & 1.02964E+00 & -0.32 \\
    3 & 7.18546E+00 & 0.25 & 1.20728E+00 & 0.75 & 6.92333E-01 & 0.57 \\
    4 & 3.72538E+00 & 0.95 & 3.48245E-01 & 1.79 & 2.03548E-01 & 1.77 \\
    5 & 1.89702E+00 & 0.97 & 9.25395E-02 & 1.91 & 5.38186E-02 & 1.92 \\
    \bottomrule
  \end{tabular}}
    \caption{Nonconforming method with non-matching mesh.}
    \label{tab:MPnonconforming_nonmatch}
\end{table}

\subsection{Iterative techniques based on domain decomposition}

Following similar procedures as in \cite{lazarov2001iterative}, this subsection applies the conjugate gradient iteration for the reduced interface problem of the coupled body-plate model. Theoretical results below indicate that the iteration count of the conjugate gradient iteration is independent of  mesh sizes. 

To derive the reduced system for the discrete problem \eqref{pro: weak-discrete-c}, introduce the following problem: Given $\bft^\alpha \in L^2(\alpha;\mathbb{R}^3)$ and  $\bft^\beta \in L^2(\beta;\mathbb{R}^3)$, find $(\tilde{\bsig}_h^{\alpha}, \tilde{\bu}_h^{\alpha}, \tilde{\bu}_h^{\beta}) \in \Sigma_h^\alpha \times V_h^{\alpha}\times W_h^\beta$ such that
\begin{equation}
\label{reduced}
\begin{aligned}
    (\mathcal{C}_0^{-1} \tilde{\bsig}_h^\alpha, \btau_h^{\alpha})_\alpha + (\divt \btau_h^{\alpha}, \tilde{\bu}_h^{\alpha})_\alpha &= 0, \\
 (\divt \tilde{\bsig}_h^{\alpha}, \bv_h^{\alpha})_\alpha &= -(\bft^{\alpha}, \bv_h^{\alpha})_\alpha, \\
     - (\bN^{\beta}(\tilde{\bu}_h^\beta), \be^{\beta}(\bv_h^\beta))_\beta - (\bM^{\beta}(\tilde{\bu}_h^\beta), \bK^{\beta}(\bv_h^\beta))_\beta &= - (\bft^{\beta}, \bv_h^{\beta})_\beta,
\end{aligned}
\end{equation}
for all $(\btau_h^\alpha,\bv_h^\alpha,\bv_h^\beta) \in \Sigma_h^\alpha \times V_h^\alpha \times W_h^\beta$. Then the solution of \eqref{pro: weak-discrete-c} can be represented as
$$
(\bsig_h^\alpha, \bu_h^\alpha, \bu_h^\beta) = ( \bar{\bsig}_h^\alpha + \tilde{\bsig}_h^\alpha, \bar{\bu}_h^\alpha+\tilde{\bu}_h^\alpha,\bar{\bu}_h^\beta+ \tilde{\bu}_h^\beta),
$$
where the remainder $( \bar{\bsig}_h^\alpha, \bar{\bu}_h^\alpha,\bar{\bu}_h^\beta)$ satisfies
\begin{equation}
\label{eq:left}
\begin{aligned}
    (\mathcal{C}_0^{-1} \bar{\bsig}_h^\alpha, \btau_h^{\alpha})_\alpha + (\divt \btau_h^{\alpha}, \bar{\bu}_h^{\alpha})_\alpha &= (\btau_h^\alpha \bn^\alpha, \bar{\bu}_h^\beta+\tilde{\bu}_h^\beta)_\Gamma, \\
 (\divt \bar{\bsig}_h^{\alpha}, \bv_h^{\alpha})_\alpha &= 0, \\
     - (\bN^{\beta}(\bar{\bu}_h^\beta), \be^{\beta}(\bv_h^\beta))_\beta - (\bM^{\beta}(\bar{\bu}_h^\beta), \bK^{\beta}(\bv_h^\beta))_\beta &= ( (\bar{\bsig}_h^\alpha + \tilde{\bsig}_h^\alpha)\bn^\alpha, \bv_h^\beta )_\Gamma,
\end{aligned}
\end{equation}
for all $(\btau_h^\alpha,\bv_h^\alpha,\bv_h^\beta) \in \Sigma_h^\alpha \times V_h^\alpha \times W_h^\beta$. Since \eqref{reduced} can be solved by two sub-equations independently, it remains to solve \eqref{eq:left}.

Introduce the discrete spaces: 
$$Q_h^{\Gamma} \coloneqq \{\bv_h^\beta|_{\Gamma};\bv_h^\beta \in  W_{h}^\beta \} \text{ and } Q_h^{\Gamma,\prime} \coloneqq \{ (\btau_h^\alpha \bn^\alpha)|_{\Gamma}; \btau_h^\alpha \in \Sigma_h^\alpha \}.$$
Define the operator
$$
E: Q_h^{\Gamma} \longrightarrow Q_h^{\Gamma,\prime} \text{ by } E \lambda = (\bsig_h^\alpha(\lambda)\bn^\alpha)|_\Gamma,
$$
where for a given $\lambda \in Q_h^\Gamma$ the pair $(\bsig^\alpha_h(\lambda), \bu^\alpha_h(\lambda)) \in \Sigma_h^\alpha \times V_h^\alpha$ is the solution of
\begin{equation*}
\begin{aligned}
    (\mathcal{C}_0^{-1} \bsig_h^\alpha(\lambda), \btau_h^{\alpha})_\alpha + (\divt \btau_h^{\alpha}, \bu_h^{\alpha}(\lambda))_\alpha &= (\btau_h^{\alpha}  \bn^{\alpha}, \lambda )_{\Gamma}, \\
 (\divt \bsig_h^{\alpha}(\lambda), \bv_h^{\alpha})_\alpha &= 0, 
\end{aligned}
\end{equation*}
for all $(\btau_h^\alpha,\bv_h^\alpha) \in \Sigma_h^\alpha \times V_h^\alpha$.
The operator $E$ is a discrete Dirichlet--Neumann mapping on $\Gamma$. Further, define the operator
$$
S: Q_h^{\Gamma,\prime} \longrightarrow Q_h^{\Gamma} \text{ by } S \lambda^* = \bu^\beta_h(\lambda^*) = (s_{2,h}(\lambda^*))|_\Gamma,$$
where for a given $\lambda^* \in Q_h^{\Gamma,\prime}$ the function $\bu^\beta_h(\lambda^*) \in Q_h^{\Gamma}$ is the solution of
\begin{equation*}
- (\bN^{\beta}(\bu_h^\beta(\lambda^*)), \be^{\beta}(\bv_h^\beta))_\beta - (\bM^{\beta}(\bu_h^\beta(\lambda^*)), \bK^{\beta}(\bv_h^\beta))_\beta = ( \lambda^*, \bv_h^{\beta})_{\Gamma} \text{ for all } \bv_h^\beta \in W_h^\beta.
\end{equation*}
The operator $S$ is a discrete Neumann--Dirichlet mapping on $\Gamma$. In terms of the operators of $S$ and $E$, it follows from \eqref{eq:left} that 
\begin{equation*}
\begin{aligned}
(\bar{\bsig}_h^\alpha \bn^\alpha)|_{\Gamma} &= E( \tilde{\bu}_h^{\beta}|_{\Gamma} +  \bar{\bu}_h^\beta|_{\Gamma} ), \\
\bar{\bu}_h^\beta|_{\Gamma}  &= S( (\tilde{\bsig}_h^\alpha\bn^\alpha) |_{\Gamma} + (\bar{\bsig}_h^\alpha \bn^\alpha) |_{\Gamma}).
\end{aligned}
\end{equation*}
Elimination of $(\bar{\bsig}_h^\alpha \bn^\alpha)|_{\Gamma}$ gives
\begin{equation}
\label{main-eq1}
(I-SE) \bar{\bu}_h^\beta|_{\Gamma} = S( (\tilde{\bsig}_h^\alpha\bn^\alpha) |_{\Gamma} + E(\tilde{\bu}_h^\beta|_{\Gamma})).    
\end{equation}

For all $(\bu^\beta,\bv^\beta) \in \bV$, define
\begin{equation*}
\label{def:a-beta}
a_{\beta}(\bu^\beta, \bv^\beta) \coloneqq (\bN^{\beta}(\bu^\beta), \be^{\beta}(\bv^\beta))_{\beta\backslash\Gamma} + (\bM^{\beta}(\bu^\beta), \bK^{\beta}(\bv^\beta))_{\beta \backslash\Gamma}.
\end{equation*}
For any $\bu_h^\beta|_{\Gamma} \in Q_h^\Gamma$, define its discrete $a_\beta$ extension $U_{\bu_h^\beta|_{\Gamma}} \in W_h^\beta$ satisfying $U_{\bu_h^\beta|_{\Gamma}}|_\Gamma = \bu_h^\beta|_{\Gamma}$ which is the unique solution to:
\begin{equation}
\label{def:harmonic-extension}
a_\beta(U_{\bu_h^\beta|_{\Gamma}}, \bw_h^\beta) = 0 \text{ for all } \bw_h^\beta \in W_h^\beta \text{ such that } \bw_h^\beta|_{\Gamma} = 0.    
\end{equation}
The energy $a_\beta(U_{\bu_h^\beta|_{\Gamma}}, U_{\bu_h^\beta|_{\Gamma}})$ defines an inner product $\langle\bu_h^\beta|_{\Gamma}, \bu_h^\beta|_{\Gamma} \rangle_{U}$ on $Q_h^\Gamma$. 

\begin{theorem}
The operator $I - SE : Q_h^\Gamma \rightarrow Q_h^\Gamma$ is symmetric and positive definite with respect to the inner product $\langle \bullet, \bullet \rangle_{U}$. Moreover, $I - SE$ has a bounded condition number independent of $h_\alpha$ and $h_\beta$.
\end{theorem}

\begin{proof}
The proof follows a similar argument to the one in \cite{lazarov2001iterative} and is omitted for brevity.  
\end{proof}

Similar results can be obtained for the nonconforming finite element method \eqref{pro: weak-discrete-nc}.  Tables \ref{tab:cg_C} and \ref{tab:cg_NC} present numerical experiments to evaluate the efficiency and robustness of the conjugate gradient solver for the reduced interface systems of the conforming finite element method and the nonconforming finite element method, respectively. In each table, the number of conjugate gradient iterations $N_\text{it}$ required to reduce the relative residual below a tolerance of $10^{-6}$, the final relative residual (Final Rel. Res.) achieved, and the average reduction factor $\rho_\text{avg}$ are reported. The convergence appears to be fairly insensitive to the mesh sizes, all in good agreement with the
theory.

\begin{table}[htbp]
  \centering
  \begin{tabular}{c ccc c ccc}
    \toprule
    & \multicolumn{3}{c}{{Matching Mesh}} & &\multicolumn{3}{c}{{Non-matching Mesh}} \\
    \cmidrule(lr){1-4} \cmidrule(lr){5-8}
     Mesh & $N_\text{it}$ & Final Rel. Res. & $\rho_\text{avg}$ & Mesh & $N_\text{it}$ & Final Rel. Res. & $\rho_\text{avg}$ \\
    \midrule
    1 & 5 & 5.48E-08 & 0.0353 & 2 & 4 & 7.11E-07 & 0.0290 \\
    2 & 5 & 5.08E-08 & 0.0348 & 3 & 5 & 1.96E-08 & 0.0287 \\
    3 & 5 & 4.64E-08 & 0.0341 & 4 & 5 & 5.86E-08 & 0.0358 \\
    \bottomrule
  \end{tabular}
    \caption{Conjugate gradient solver for conforming method.}
      \label{tab:cg_C}
\end{table}

\begin{table}[htbp]
  \centering
  \begin{tabular}{c ccc c ccc}
    \toprule
    & \multicolumn{3}{c}{{Matching Mesh}} & &\multicolumn{3}{c}{{Non-matching Mesh}} \\
    \cmidrule(lr){1-4} \cmidrule(lr){5-8}
     Mesh & $N_\text{it}$ & Final Rel. Res. & $\rho_\text{avg}$ & Mesh & $N_\text{it}$ & Final Rel. Res. & $\rho_\text{avg}$ \\
    \midrule
    1 & 4 & 5.54E-07 & 0.0273 & 3 & 4 & 3.30E-07 & 0.0240 \\
    2 & 5 & 1.61E-08 & 0.0276 & 4 & 5 & 2.16E-08 & 0.0293 \\
    3 & 4 & 7.94E-07 & 0.0299 & 5 & 4 & 7.79E-07 & 0.0297 \\
    \bottomrule
  \end{tabular}
   \caption{Conjugate gradient solver for nonconforming method.}
     \label{tab:cg_NC}
\end{table}

\clearpage
\bibliographystyle{siamplain}
\bibliography{references}

\begin{thebibliography}{10}

\bibitem{adini1960analysis}
{\sc A.~Adini and R.~W. Clough}, {\em Analysis of plate bending by the finite
  element method}, University of California, 1960.

\bibitem{arango1998some}
{\sc J.~Arango, L.~Lebedev, and I.~Vorovich}, {\em Some boundary value problems
  and models for coupled elastic bodies}, Quart. Appl. Math., 56 (1998),
  pp.~157--172.

\bibitem{Argyris1968}
{\sc J.~Argyris, I.~Fried, and D.~W. Scharpf}, {\em The tuba family of plate
  elements for the matrix displacement method}, Aeronaut. J., 72 (1968),
  pp.~701--709.

\bibitem{arnold2005rectangular}
{\sc D.~Arnold and G.~Awanou}, {\em Rectangular mixed finite elements for
  elasticity}, Math. Models Methods Appl. Sci., 15 (2005), pp.~1417--1429.

\bibitem{arnold2008finite}
{\sc D.~Arnold, G.~Awanou, and R.~Winther}, {\em Finite elements for symmetric
  tensors in three dimensions}, Math. Comp., 77 (2008), pp.~1229--1251.

\bibitem{ArnoldWinther2002}
{\sc D.~Arnold and R.~Winther}, {\em Mixed finite elements for elasticity},
  Numer. Math., 92 (2002), pp.~401--419.

\bibitem{arnold2014nonconforming}
{\sc D.~N. Arnold, G.~Awanou, and R.~Winther}, {\em Nonconforming tetrahedral
  mixed finite elements for elasticity}, Math. Models Methods Appl. Sci., 24
  (2014), pp.~783--796.

\bibitem{arnold1985mixed}
{\sc D.~N. Arnold and F.~Brezzi}, {\em Mixed and nonconforming finite element
  methods: implementation, postprocessing and error estimates}, ESAIM Math.
  Model. Numer. Anal., 19 (1985), pp.~7--32.

\bibitem{aufranc1989numerical}
{\sc M.~Aufranc}, {\em Numerical study of a junction between a
  three-dimensional elastic structure and a plate}, Comput. Methods Appl. Mech.
  Engrg., 74 (1989), pp.~207--222.

\bibitem{boffi2013mixed}
{\sc D.~Boffi, F.~Brezzi, and M.~Fortin}, {\em Mixed finite element methods and
  applications}, vol.~44, Springer, 2013.

\bibitem{bournival2010mesh}
{\sc S.~Bournival, J.-C. Cuilli{\`e}re, and V.~Fran{\c{c}}ois}, {\em A
  mesh-geometry based method for coupling 1d and 3d elements}, Adv. Eng.
  Softw., 41 (2010), pp.~838--858.

\bibitem{brenner2008mathematical}
{\sc S.~C. Brenner}, {\em The mathematical theory of finite element methods},
  Springer, 2008.

\bibitem{chen2009ap}
{\sc C.~Chen, J.~Huang, and X.~Huang}, {\em Ap 1-p 3-nzt fem for solving
  general elastic multi-structure problems.}, J. Comput. Anal. Appl., 11
  (2009).

\bibitem{Ciarlet1989JMPA}
{\sc P.~G. Ciarlet, H.~Le~Dret, and R.~Nzengwa}, {\em Junctions between
  three-dimensional and two-dimensional linearly elastic structures}, J. Math.
  Pures Appl., 68 (1989), pp.~261--295.

\bibitem{Fengshi1996}
{\sc K.~Feng and Z.~Shi}, {\em Mathematical theory of elastic structures},
  Springer-Verlag, Berlin; Science Press Beijing, Beijing, 1996.
\newblock Translated from the 1981 Chinese original, Revised by the authors.

\bibitem{gopalakrishnan2011symmetric}
{\sc J.~Gopalakrishnan and J.~Guzm{\'a}n}, {\em Symmetric nonconforming mixed
  finite elements for linear elasticity}, SIAM J. Numer. Anal., 49 (2011),
  pp.~1504--1520.

\bibitem{gudi2010new}
{\sc T.~Gudi}, {\em A new error analysis for discontinuous finite element
  methods for linear elliptic problems}, Math. Comp., 79 (2010),
  pp.~2169--2189.

\bibitem{huang2011NMPDE}
{\sc L.~Guo and J.~Huang}, {\em Adini-{$Q_1$}-{$P_3$} {FEM} for general elastic
  multi-structure problems}, Numer. Methods Partial Differential Equ., 27
  (2011), pp.~1092--1112.

\bibitem{hansbo2005nitsche}
{\sc P.~Hansbo}, {\em Nitsche's method for interface problems in computa-tional
  mechanics}, GAMM-Mitteilungen, 28 (2005), pp.~183--206.

\bibitem{hansbo2022nitsche}
{\sc P.~Hansbo and M.~G. Larson}, {\em Nitsche’s finite element method for
  model coupling in elasticity}, Comput. Methods Appl. Mech. Engrg., 392
  (2022), p.~114707.

\bibitem{Hu2015}
{\sc J.~Hu}, {\em Finite element approximations of symmetric tensors on
  simplicial grids in {$\mathbb R^n$}: the higher order case}, J. Comput.
  Math., 33 (2015), pp.~283--296.

\bibitem{Hu20152}
{\sc J.~Hu}, {\em A new family of efficient conforming mixed finite elements on
  both rectangular and cuboid meshes for linear elasticity in the symmetric
  formulation}, SIAM J. Numer. Anal., 53 (2015), pp.~1438--1463.

\bibitem{hu2025mixed}
{\sc J.~Hu, Z.~Liu, R.~Ma, and R.~Wang}, {\em A mixed finite element method for
  coupled plates}, Comput. Methods Appl. Math., 25 (2025), pp.~601--618.

\bibitem{hu2024direct}
{\sc J.~Hu and L.~Ma}, {\em A direct finite element method for elliptic
  interface problems}, arXiv preprint arXiv:2401.16967v2,  (2025).

\bibitem{HuMa2018}
{\sc J.~Hu and R.~Ma}, {\em Conforming mixed triangular prism elements for the
  linear elasticity problem}, Int. J. Numer. Anal. Model., 15 (2018),
  pp.~228--242.

\bibitem{hu2014new}
{\sc J.~Hu, R.~Ma, and Z.~Shi}, {\em A new a priori error estimate of
  nonconforming finite element methods}, Sci. China Math., 57 (2014),
  pp.~887--902.

\bibitem{hu2016simplest}
{\sc J.~Hu, H.~Man, J.~Wang, and S.~Zhang}, {\em The simplest nonconforming
  mixed finite element method for linear elasticity in the symmetric
  formulation on n-rectangular grids}, Comput. Math. Appl., 71 (2016),
  pp.~1317--1336.

\bibitem{hu2014simple}
{\sc J.~Hu, H.~Man, and S.~Zhang}, {\em A simple conforming mixed finite
  element for linear elasticity on rectangular grids in any space dimension},
  J. Sci. Comput., 58 (2014), pp.~367--379.

\bibitem{HuZhang2015}
{\sc J.~Hu and S.~Zhang}, {\em A family of symmetric mixed finite elements for
  linear elasticity on tetrahedral grids}, Sci. China Math., 58 (2015),
  pp.~297--307.

\bibitem{huang2004numerical}
{\sc J.~Huang}, {\em Numerical solution of the elastic body-plate problem by
  nonoverlapping domain decomposition type techniques}, Math. Comp., 73 (2004),
  pp.~19--34.

\bibitem{huang2007finite}
{\sc J.~Huang, L.~Guo, and Z.~Shi}, {\em A finite element method for
  investigating general elastic multi-structures}, Comput. Math. Appl., 53
  (2007), pp.~1867--1895.

\bibitem{jianguo2005some}
{\sc J.~Huang, Z.~Shi, and Y.~Xu}, {\em Some studies on mathematical models for
  general elastic multi-structures}, Sci. China Ser. A, 48 (2005),
  pp.~986--1007.

\bibitem{huang2006finite}
{\sc J.~Huang, Z.~Shi, and Y.~Xu}, {\em Finite element analysis for general
  elastic multi-structures}, Sci. China Ser. A, 49 (2006), p.~109.

\bibitem{klarmann2022coupling}
{\sc S.~Klarmann, J.~Wackerfu{\ss}, and S.~Klinkel}, {\em Coupling 2d continuum
  and beam elements: a mixed formulation for avoiding spurious stresses},
  Comput. Mech., 70 (2022), pp.~1145--1166.

\bibitem{lazarov2001iterative}
{\sc R.~D. Lazarov, J.~E. Pasciak, and P.~S. Vassilevski}, {\em Iterative
  solution of a coupled mixed and standard galerkin discretization method for
  elliptic problems}, Numer. Linear Algebra Appl., 8 (2001), pp.~13--31.

\bibitem{li2014new}
{\sc M.~Li, X.~Guan, and S.~Mao}, {\em New error estimates of the morley
  element for the plate bending problems}, J. Comput. Appl. Math., 263 (2014),
  pp.~405--416.

\bibitem{lions2012non}
{\sc J.~L. Lions and E.~Magenes}, {\em Non-homogeneous boundary value problems
  and applications: Vol. 1}, vol.~181, Springer Science \& Business Media,
  2012.

\bibitem{ljulj20193d}
{\sc M.~Ljulj and J.~Tamba{\v{c}}a}, {\em 3d structure--2d plate interaction
  model}, Math. Mech. Solids, 24 (2019), pp.~3354--3377.

\bibitem{ljulj2021numerical}
{\sc M.~Ljulj and J.~Tamba{\v{c}}a}, {\em Numerical investigation of the 2d--1d
  structure interaction model}, Math. Mech. Solids, 26 (2021), pp.~1876--1895.

\bibitem{ming2006morley}
{\sc W.~Ming and J.~Xu}, {\em The morley element for fourth order elliptic
  equations in any dimensions}, Numer. Math., 103 (2006), pp.~155--169.

\bibitem{nguyen2014nitsche}
{\sc V.~Nguyen, P.~Kerfriden, S.~Claus, and S.~BORDAS}, {\em Nitsche’s method
  method for mixed dimensional analysis: conforming and non-conforming
  continuum-beam and continuum-plate coupling}, Comput. Methods Appl. Mech.
  Engrg.,  (2014).

\bibitem{pauly2022hilbert}
{\sc D.~Pauly and M.~Schomburg}, {\em Hilbert complexes with mixed boundary
  conditions—part 2: Elasticity complex}, Math. Methods Appl. Sci., 45
  (2022), pp.~8971--9005.

\bibitem{scott1990finite}
{\sc L.~R. Scott and S.~Zhang}, {\em Finite element interpolation of nonsmooth
  functions satisfying boundary conditions}, Math. Comp., 54 (1990),
  pp.~483--493.

\bibitem{shi1990error}
{\sc Z.~C. Shi}, {\em Error estimates of morley element}, Math. Numer. Sinica,
  12 (1990), pp.~113--118.

\bibitem{shim2002mixed}
{\sc K.~W. Shim, D.~J. Monaghan, and C.~G. Armstrong}, {\em Mixed dimensional
  coupling in finite element stress analysis}, Eng. Comput., 18 (2002),
  pp.~241--252.

\bibitem{song2010rigorous}
{\sc H.~Song and D.~Hodges}, {\em Rigorous joining of advanced
  reduced-dimensional beam models to 2-d finite element models}, in 51st
  AIAA/ASME/ASCE/AHS/ASC Structures, Structural Dynamics, and Materials
  Conference 18th AIAA/ASME/AHS Adaptive Structures Conference 12th, 2010,
  p.~2545.

\bibitem{steinbrecher2020mortar}
{\sc I.~Steinbrecher, M.~Mayr, M.~J. Grill, J.~Kremheller, C.~Meier, and
  A.~Popp}, {\em A mortar-type finite element approach for embedding 1d beams
  into 3d solid volumes}, Comput. Mech., 66 (2020), pp.~1377--1398.

\bibitem{veeser2019quasi}
{\sc A.~Veeser and P.~Zanotti}, {\em Quasi-optimal nonconforming methods for
  symmetric elliptic problems. ii---overconsistency and classical nonconforming
  elements}, SIAM J. Numer. Anal., 57 (2019), pp.~266--292.

\bibitem{wang1992mathematical}
{\sc L.~Wang}, {\em The mathematical models of some composite elastic
  structures and their finite element approximations}, Numer. Math. Sinica, 14
  (1992), pp.~358--370.

\bibitem{wieners1998coupling}
{\sc C.~Wieners and B.~I. Wohlmuth}, {\em The coupling of mixed and conforming
  finite element discretizations}, Contemp. Math., 218 (1998), p.~547.

\bibitem{yamamoto2019numerical}
{\sc T.~Yamamoto, T.~Yamada, and K.~Matsui}, {\em Numerical procedure to couple
  shell to solid elements by using nitsche’s method}, Comput. Mech., 63
  (2019), pp.~69--98.

\end{thebibliography}
\end{document}